\documentclass[10pt]{amsart}
\usepackage{amsmath}
\usepackage{amsxtra}
\usepackage{amscd}
\usepackage{amsthm}
\usepackage{amsfonts}
\usepackage{amssymb}
\usepackage{eucal}

\newtheorem{cor}[subsection]{Corollary}
\newtheorem{lem}[subsection]{Lemma}
\newtheorem{prop}[subsection]{Proposition}

\newtheorem{conj}[subsection]{Conjecture}
\newtheorem{thm}[subsection]{Theorem}

\theoremstyle{definition}

\theoremstyle{remark}

\newcommand{\thmref}[1]{Theorem~\ref{#1}}
\newcommand{\secref}[1]{Sect.~\ref{#1}}
\newcommand{\lemref}[1]{Lemma~\ref{#1}}
\newcommand{\propref}[1]{Proposition~\ref{#1}}
\newcommand{\corref}[1]{Corollary~\ref{#1}}
\newcommand{\conjref}[1]{Conjecture~\ref{#1}}

\numberwithin{equation}{section}

\newcommand{\nc}{\newcommand}
\nc{\renc}{\renewcommand}
\nc{\ssec}{\subsection}
\nc{\sssec}{\subsubsection}
\nc{\on}{\operatorname}

\nc\ol{\overline}
\nc\wt{\widetilde}
\nc\tboxtimes{\wt{\boxtimes}}
\nc{\alp}{\alpha}

\nc{\Fq}{{\mathbb F}_q}
\nc{\Fqb}{\ol{{\mathbb F}_q}}
\nc{\Ql}{\ol{{\mathbb Q}_\ell}}
\nc{\id}{\text{id}}

\nc{\Hom}{\on{Hom}}
\nc{\Lie}{\on{Lie}}
\nc{\Loc}{\on{Loc}}
\nc{\Pic}{\on{Pic}}
\nc{\Bun}{\on{Bun}}
\nc{\IC}{\on{IC}}
\nc{\Aut}{\on{Aut}}
\nc{\pos}{{\on{pos}}}
\nc{\Conv}{\on{Conv}}
\nc{\Sph}{\on{Sph}}
\nc{\Sym}{\on{Sym}}

\nc{\BunBb}{\overline{\Bun}_B}
\nc{\Buno}{\overset{\circ}{\Bun}}
\nc{\BunPb}{{\overline{\Bun}_P}}
\nc{\BunBM}{\Bun_{B(M)}}
\nc{\BunBm}{\Bun_{B^-}}
\nc{\BunBmb}{\ol{\Bun}_{B^-}}
\nc{\BunBMb}{\overline{\Bun}_{B(M)}}
\nc{\BunPbw}{{\widetilde{\Bun}_P}}
\nc{\BunBP}{\widetilde{\Bun}_{B,P}}
\nc{\GUb}{\overline{G/U}}
\nc{\GUPb}{\overline{G/U(P)}}

\nc{\oX}{\overset{\circ}{X}{}}
\nc{\oFS}{\overset{\circ}{\on{FS}}{}}
\nc{\hl}{\overset{\leftarrow}h{}}
\nc{\hr}{\overset{\rightarrow}h{}}
\nc{\D}{{\mathcal D}}
\nc{\Gr}{{\on{Gr}}}

\nc\Spec{\on{Spec}}
\nc\Mod{\on{Mod}}

\nc{\lambdach}{{\check\lambda}}
\nc{\Lambdach}{{\check\Lambda}{}}
\nc{\much}{{\check\mu}}
\nc{\omegach}{{\check\omega}}
\nc{\nuch}{{\check\nu}}
\nc{\etach}{{\check\eta}}
\nc{\alphach}{{\check\alpha}}
\nc{\betach}{{\check\beta}}
\nc{\rhoch}{{\check\rho}}

\nc{\BA}{{\mathbb{A}}}
\nc{\BC}{{\mathbb{C}}}
\nc{\BG}{{\mathbb{G}}}
\nc{\BM}{{\mathbb{M}}}
\nc{\BD}{{\mathbb{D}}}
\nc{\BN}{{\mathbb{N}}}
\nc{\BP}{{\mathbb{P}}}
\nc{\BR}{{\mathbb{R}}}
\nc{\BZ}{{\mathbb{Z}}}
\nc{\BS}{{\mathbb{S}}}
\nc{\BQ}{{\mathbb{Q}}}

\nc{\CA}{{\mathcal{A}}}
\nc{\CB}{{\mathcal{B}}}

\nc{\CE}{{\mathcal{E}}}
\nc{\CF}{{\mathcal{F}}}

\nc{\CL}{{\mathcal{L}}}
\nc{\CC}{{\mathcal{C}}}
\nc{\CM}{{\mathcal{M}}}
\nc{\CN}{{\mathcal{N}}}
\nc{\CK}{{\mathcal{K}}}
\nc{\CO}{{\mathcal{O}}}
\nc{\CP}{{\mathcal{P}}}
\nc{\CQ}{{\mathcal{Q}}}
\nc{\CR}{{\mathcal{R}}}
\nc{\CS}{{\mathcal{S}}}
\nc{\CT}{{\mathcal{T}}}
\nc{\CU}{{\mathcal{U}}}
\nc{\CV}{{\mathcal{V}}}
\nc{\CW}{{\mathcal{W}}}
\nc{\CZ}{{\mathcal{Z}}}
\nc{\CX}{{\mathcal{X}}}
\nc{\CY}{{\mathcal{Y}}}
\nc{\CI}{{\mathcal{I}}}

\nc{\csM}{{\check{\mathcal A}}{}}
\nc{\oM}{{\overset{\circ}{\mathcal M}}{}}
\nc{\obM}{{\overset{\circ}{\mathbf M}}{}}
\nc{\oCA}{{\overset{\circ}{\mathcal A}}{}}
\nc{\obA}{{\overset{\circ}{\mathbf A}}{}}
\nc{\ooM}{{\overset{\circ}{M}}{}}
\nc{\osM}{{\overset{\circ}{\mathsf M}}{}}
\nc{\vM}{{\overset{\bullet}{\mathcal M}}{}}
\nc{\nM}{{\underset{\bullet}{\mathcal M}}{}}
\nc{\oD}{{\overset{\circ}{\mathcal D}}{}}
\nc{\obD}{{\overset{\circ}{\mathbf D}}{}}
\nc{\oA}{{\overset{\circ}{\mathbb A}}{}}
\nc{\op}{{\overset{\bullet}{\mathbf p}}{}}
\nc{\cp}{{\overset{\circ}{\mathbf p}}{}}
\nc{\oU}{{\overset{\bullet}{\mathcal U}}{}}
\nc{\oZ}{{\overset{\circ}{\mathcal Z}}{}}
\nc{\oL}{{\overset{\circ}{\mathcal L}}{}}
\nc{\ofZ}{{\overset{\circ}{\mathfrak Z}}{}}
\nc{\oF}{{\overset{\circ}{\fF}}}
\nc{\oG}{{\overset{\circ}{G}}}

\nc{\fa}{{\mathfrak{a}}}
\nc{\fb}{{\mathfrak{b}}}
\nc{\fd}{{\mathfrak{d}}}
\nc{\fg}{{\mathfrak{g}}}
\nc{\fgl}{{\mathfrak{gl}}}
\nc{\fh}{{\mathfrak{h}}}
\nc{\fj}{{\mathfrak{j}}}
\nc{\fl}{{\mathfrak{l}}}
\nc{\fm}{{\mathfrak{m}}}
\nc{\fn}{{\mathfrak{n}}}
\nc{\fu}{{\mathfrak{u}}}
\nc{\fp}{{\mathfrak{p}}}
\nc{\fr}{{\mathfrak{r}}}
\nc{\fs}{{\mathfrak{s}}}
\nc{\hsl}{{\widehat{\mathfrak{sl}}}}
\nc{\hgl}{{\widehat{\mathfrak{gl}}}}
\nc{\hg}{{\widehat{\mathfrak{g}}}}
\nc{\chg}{{\widehat{\mathfrak{g}}}{}^\vee}
\nc{\hn}{{\widehat{\mathfrak{n}}}}
\nc{\chn}{{\widehat{\mathfrak{n}}}{}^\vee}

\nc{\fA}{{\mathfrak{A}}}
\nc{\fB}{{\mathfrak{B}}}
\nc{\fD}{{\mathfrak{D}}}
\nc{\fE}{{\mathfrak{E}}}
\nc{\fF}{{\mathfrak{F}}}
\nc{\fG}{{\mathfrak{G}}}
\nc{\fK}{{\mathfrak{K}}}
\nc{\fL}{{\mathfrak{L}}}
\nc{\fM}{{\mathfrak{M}}}
\nc{\fN}{{\mathfrak{N}}}
\nc{\fP}{{\mathfrak{P}}}
\nc{\fU}{{\mathfrak{U}}}
\nc{\fV}{{\mathfrak{V}}}
\nc{\fZ}{{\mathfrak{Z}}}
\nc{\fW}{{\mathfrak{W}}}

\nc{\oP}{{\overset{\circ}{\fP}}}

\nc{\bb}{{\mathbf{b}}}
\nc{\bc}{{\mathbf{c}}}
\nc{\bd}{{\mathbf{d}}}
\nc{\be}{{\mathbf{e}}}
\nc{\bj}{{\mathbf{j}}}
\nc{\bn}{{\mathbf{n}}}
\nc{\bp}{{\mathbf{p}}}
\nc{\bq}{{\mathbf{q}}}
\nc{\bu}{{\mathbf{u}}}
\nc{\bv}{{\mathbf{v}}}
\nc{\bx}{{\mathbf{x}}}
\nc{\bs}{{\mathbf{s}}}
\nc{\by}{{\mathbf{y}}}
\nc{\bw}{{\mathbf{w}}}
\nc{\bA}{{\mathbf{A}}}
\nc{\bK}{{\mathbf{K}}}
\nc{\bB}{{\mathbf{B}}}
\nc{\bU}{{\mathbf{U}}}
\nc{\bC}{{\mathbf{C}}}
\nc{\bG}{{\mathbf{G}}}
\nc{\bD}{{\mathbf{D}}}
\nc{\bH}{{\mathbf{H}}}
\nc{\bM}{{\mathbf{M}}}
\nc{\bN}{{\mathbf{N}}}
\nc{\bV}{{\mathbf{V}}}
\nc{\bW}{{\mathbf{W}}}
\nc{\bX}{{\mathbf{X}}}
\nc{\bZ}{{\mathbf{Z}}}
\nc{\bS}{{\mathbf{S}}}

\nc{\sA}{{\mathsf{A}}}
\nc{\sB}{{\mathsf{B}}}
\nc{\sC}{{\mathsf{C}}}
\nc{\sD}{{\mathsf{D}}}
\nc{\sF}{{\mathsf{F}}}
\nc{\sK}{{\mathsf{K}}}
\nc{\sM}{{\mathsf{M}}}
\nc{\sO}{{\mathsf{O}}}
\nc{\sW}{{\mathsf{W}}}
\nc{\sQ}{{\mathsf{Q}}}
\nc{\sG}{{\mathsf{G}}}
\nc{\sP}{{\mathsf{P}}}
\nc{\sZ}{{\mathsf{Z}}}
\nc{\sfp}{{\mathsf{p}}}
\nc{\sr}{{\mathsf{r}}}
\nc{\sk}{{\mathsf{k}}}
\nc{\sg}{{\mathsf{g}}}
\nc{\sff}{{\mathsf{f}}}
\nc{\sfb}{{\mathsf{b}}}
\nc{\sfc}{{\mathsf{c}}}
\nc{\sd}{{\mathsf{d}}}

\nc{\BK}{{\bar{K}}}

\nc{\tA}{{\widetilde{\mathbf{A}}}}
\nc{\tB}{{\widetilde{\mathcal{B}}}}
\nc{\tg}{{\widetilde{\mathfrak{g}}}}
\nc{\tG}{{\widetilde{G}}}
\nc{\TM}{{\widetilde{\mathbb{M}}}{}}
\nc{\tO}{{\widetilde{\mathsf{O}}}{}}
\nc{\tU}{{\widetilde{\mathfrak{U}}}{}}
\nc{\TZ}{{\tilde{Z}}}
\nc{\tx}{{\tilde{x}}}
\nc{\tbv}{{\tilde{\bv}}}
\nc{\tfP}{{\widetilde{\mathfrak{P}}}{}}
\nc{\tz}{{\tilde{\zeta}}}
\nc{\tmu}{{\tilde{\mu}}}

\nc{\urho}{\underline{\rho}}
\nc{\uB}{\underline{B}}
\nc{\uC}{{\underline{\mathbb{C}}}}
\nc{\ui}{\underline{i}}
\nc{\uj}{\underline{j}}
\nc{\ofP}{{\overline{\mathfrak{P}}}}
\nc{\oB}{{\overline{\mathcal{B}}}}
\nc{\og}{{\overline{\mathfrak{g}}}}
\nc{\oI}{{\overline{I}}}

\nc{\eps}{\varepsilon}
\nc{\hrho}{{\hat{\rho}}}

\nc{\one}{{\mathbf{1}}}
\nc{\two}{{\mathbf{t}}}

\nc{\Rep}{{\mathop{\operatorname{\rm Rep}}}}
\nc{\End}{{\mathop{\operatorname{\rm End}}}}
\nc{\Ext}{{\mathop{\operatorname{\rm Ext}}}}
\nc{\length}{{\mathop{\operatorname{\rm length}}}}
\nc{\supp}{{\mathop{\operatorname{\rm supp}}}}
\renc{\mod}{{\text{-mod}}}

\nc{\Fl}{\on{Fl}}

\nc{\reg}{{\text{\rm reg}}}

\nc\wh{\widehat}

\nc{\fq}{\mathfrak q}

\nc{\fqb}{\ol{\fq}{}}
\nc{\fpb}{\ol{\fp}{}}

\nc{\hattimes}{\wh\otimes}
\nc{\FS}{\on{FS}}

\nc{\bh}{{\bar{h}}}
\nc{\bOmega}{{\overline{\Omega(\check \fn)}}}

\nc{\seq}[1]{\stackrel{#1}{\sim}}

\nc{\oCY}{\overset{\circ}{\CY}}
\nc{\oZZ}{\overset{\circ}{Z}}

\nc{\cT}{{\check{T}}}
\nc{\cG}{{\check{G}}}
\nc{\cM}{{\check{M}}}
\nc{\cB}{{\check{B}}}

\nc{\ct}{{\check{\mathfrak t}}}
\nc{\cg}{{\check{\fg}}}
\nc{\cb}{{\check{\fb}}}
\nc{\cn}{{\check{\fn}}}

\nc{\cLambda}{{\check\Lambda}}

\nc{\cla}{{\check\lambda}}
\nc{\cmu}{{\check\mu}}
\nc{\cnu}{{\check\nu}}
\nc{\ceta}{{\check\eta}}
\nc{\crho}{{\check\rho}}

\nc{\omegah}{{\omega^{\frac{1}{2}}}}
\nc{\omegacrho}{{\omega^{\crho}}}

\nc{\imathb}{{\ol{\imath}}}

\nc{\Whit}{\on{Whit}}
\nc{\oy}{\overline{y}}
\nc{\ox}{{\overline{x}}}
\nc{\mer}{{\on{mer}}}

\nc{\ssf}{\mathsf f}
\nc{\fqbm}{\ol\fq^-}
\nc{\fpbm}{\ol\fp^-}

\nc{\Const}{\mathsf{Const}}
\nc{\ppart}{(\!(t)\!)}

\nc{\KL}{\on{KL}}

\begin{document}

\title{Twisted Whittaker model and factorizable sheaves}

\author{D.~Gaitsgory}

\address{Department of Mathematics, Harvard University, Cambridge MA 02138 (USA)}

\email{gaitsgde@math.harvard.edu}

\date{May 29, 2007; revised December 25, 2007}

\maketitle

\section*{Introduction}

\ssec{Quantum Langlands duality?}

The idea behind the paper is the quest for Langlands duality for quantum groups.
Let us explain what we mean by this. 

Let $G$ be a simple algebraic group over $\BC$.
Recall that the geometric Satake equivalence realizes the Langlands dual
group $\check{G}$, or rather the category of its representations, explicitly
in terms of $G$ as follows. We consider the affine Grassmannian $\Gr_G=
G\ppart/G[[t]]$ as an ind-scheme acted on (from the left) by $G[[t]]$ and
let $\fD\mod(\Gr_G)^{G[[t]]}$ denote the corresponding category of D-modules.

One endows $\fD\mod(\Gr_G)^{G[[t]]}$ with a monoidal structure, and one
shows that it has a natural commutativity constraint. It is a basic fact,
conjectured by Drinfeld based on an earlier work of Lusztig, and proved
in \cite{Gi} and \cite{MV}, that that
the resulting tensor category is equivalent to $\Rep(\cG)$--the category of
representations of $\cG$ as an algebraic group.

\medskip

Now, $\Rep(\cG)$, considered as a braided monoidal category, admits a 
one-parameter family of deformations to the category $\Rep(U_q(\cG))$,
where by $U_q(\cG)$ we denote the quantum group attached to $G$
(when $q$ is a root of unity, we take Lusztig's quantum group with
$q$-divided powers). It has always been very tempting to try to realize
$\Rep(U_q(\cG))$ also via $\Gr_G$, or some closely related geometric
object, as a category of D-modules or perverse sheaves with a particular
equivariance property.

One has a natural candidate of how to involve the 
parameter $q$. Namely, let $\wt\Gr_G$ be the canonical line bundle
over $G$, i.e., $\wt\Gr_G:=\widehat{G}/G[[t]]$, where $\widehat{G}$ is the
Kac-Moody extension of the loop group $G\ppart$. 
The passage from $\check{G}$ to its quantum deformation should correspond
to replacing D-modules (or perverse sheaves) on $\Gr_G$ by 
D-modules (or perverse sheaves) on $\wt\Gr_G$, which are monodromic 
along the fiber with monodromy $q^2$, i.e., we will consider the corresponding 
category of twisted D-modules on $\Gr_G$ which we will denote by 
$\fD\mod^c(\wt\Gr_G)$ with $q=exp(\pi i c)$.

However, the first naive attempt, i.e., to consider the category 
$\fD\mod^c(\wt\Gr_G)^{G[[t]]}$ leads to a wrong answer. E.g., when $c$
is irrational (i.e.,, when $q$
is not a root of unity) the latter category will have only one irreducible
object, i.e., it cannot be a deformation of $\Rep(\cG)$.

\ssec{Whittaker category}   \label{Whit trouble}

A viable candidate for an appropriate category of twisted D-modules 
on $\Gr_G$ was suggested by Jacob Lurie in October 2006. To explain
his idea, we will first replace the category $\fD\mod(\Gr_G)^{G[[t]]}$
by a different, but still equivalent, category, which, however, admits
a natural quantum deformation.

Namely, let $N$ be a maximal unipotent subgroup of $G$,
and consider the corresponding loop group $N\ppart$. Let
$\chi:N\ppart\to \BG_a$ be a non-degenerate character, normalized
to have conductor $0$. Let us consider the category of D-modules
on $\Gr_G$, equivariant with respect to $N\ppart$ against the
character $\chi$. We shall refer to this category as that of
Whittaker D-modules on $\Gr_G$, and denote it $\Whit(\Gr_G)$,
or simply $\Whit$.

The trouble is, however, that the orbits of the group $N\ppart$ on
$\Gr_G$ are
infinite-dimensional, and if one understands the $N\ppart$-equivariance
condition naively, the category $\Whit(\Gr_G)$ will be empty. There are
two ways to overcome this difficulty:

\medskip

One way is to try to define $\Whit(\Gr_G)$ as a triangulated category,
and then extract a t-structure. This approach has not been worked
out yet, but J.~Lurie and the author hope that this is feasible, and 
leads to a manageable theory.

\medskip

Another approach, which we opt for in the present paper,
was developed in \cite{FGV} and \cite{Ga}, where one
replaces $\Gr_G$ with the group $N\ppart$ acting on it, by a different 
geometric object, which is supposed to mimic its properties. This other
object is the Drinfeld compactification, denoted in this paper by $\fW$,
and it involves a choice of a projective curve $X$. In addition to being less
natural, the category $\Whit(\Gr_G)$, defined in this way, has the
main disadvantage that its local nature (we think of the ring $\BC[[t]]$
and the field $\BC\ppart$ as associated to a formal disc of a point $x$ on
a not necessarily complete curve $X$) is not obvious.

We will explain the definition of $\Whit(\Gr_G)$ via $\fW$ in \secref{surrogate}
and \secref{def cat}. For the purposes of the introduction we will pretend
that the first approach mentioned above works, i.e., that we have a direct
local definition of $\Whit(\Gr_G):=\fD\mod(\Gr_G)^{N\ppart,\chi}$.

\medskip

Having "defined" $\Whit(\Gr_G)$, the main theorem of the paper \cite{FGV}
can be rephrased as follows: there exists an equivalence of {\it abelian}
categories $\fD\mod(\Gr_G)^{G[[t]]}\simeq \Whit(\Gr_G)$. Moreover,
a naturally defined triangulated category, whose core is $\Whit(\Gr_G)$,
is semi-simple, i.e., we have an equivalence of abelian categories
\begin{equation} \label{cl equiv}
\Whit(\Gr_G)\simeq \Rep(\cG),
\end{equation}
which extends to an equivalence of the corresponding triangulated categories.

We emphasize that the latter statement regarding the original equivalence
\begin{equation} \label{satake equiv}
\fD\mod(\Gr_G)^{G[[t]]}\simeq \Rep(\cG)
\end{equation}
would be false: the $G[[t]]$-equivariant derived category of D-modules
on $\Gr_G$ is not at all semi-simple. This can be viewed as another
reason why \eqref{satake equiv} does not have a quantum deformation.

\ssec{Lurie's category}

Now we can explain Jacob Lurie's idea. By the same token as we "define"
$\Whit(\Gr_G)$, we can define its twisted version
$$\Whit^c(\Gr_G):=\fD\mod^c(\Gr_G)^{N\ppart,\chi}.$$
This makes sense since the Kac-Moody extension $\widehat{G}$ canonically 
splits over $N\ppart$. His conjecture can be stated as
\begin{conj}  \label{Lurie}
There exists an equivalence 
$$\Whit^c(\Gr_G)\simeq \Rep(U_q(\cG)),\,\, q=exp(\pi i c).$$
\end{conj}

In this paper we will essentially prove this conjecture for $q$ being
not a root of unity.

\medskip

Let us note that the assertion of the above conjecture 
is inherently transcendental, due to the appearance of
the exponential function relating the parameters on both
sides. We will cure this as follows:

When $q$ is not a root of unity, i.e., when $c$ is irrational, we will eventually
replace the RHS, i.e., $\Rep(U_q(\cG))$, by another category, namely that
of factorizable sheaves, denoted $\FS^c(\cG)$, which will also be equivalent 
to $\Rep(U_q(\cG))$ by a transcendental procedure. In the present paper
we will establish an equivalence
\begin{equation} \label{Whit FS}
\Whit^c(\Gr_G)\simeq \FS^c(\cG),
\end{equation}
which is algebraic; in particular, both sides and the equivalence between
them are defined over an arbitrary ground field of characteristic $0$.

At present we do not know how to modify the definition of
$\FS^c(\cG)$ to obtain a category, equivalent to $\Rep(U_q(\cG))$ for
all $q$. \footnote{A candidate for a such a modification of $\FS^c(\cG)$
has been found when this paper was under revision in December 2007.}
There is, however, another category with this property, see \secref{KL}.

\ssec{Chiral categories}   \label{chiral categories}

Returning to the statement of \conjref{Lurie}, we should explain what are
the additional structures on both sides, which are supposed to be respected
by the conjectural equivalence. (If for generic $q$ we just kept the abelian
category structure on both sides, the statement of the conjecture would not
be very interesting, as both LHS and RHS are semi-simple categories with
naturally identified sets of irreducible objects.)

\medskip

The RHS of the equivalence, i.e.,  $\Rep(U_q(\cG))$, is a braided
monoidal category. The LHS is supposed to have another kind of structure, that
we shall call "fusion" or "chiral" category. The notion of chiral category
is a subject of another work in progress of Jacob Lurie and the author.

Unfortunately, in the present global 
definition of $\Whit^c(\Gr_G)$, the chiral category structure on it is not 
evident; therefore we do not give a formal definition in the main body of 
the paper. Let us, nonetheless, indicate it here.

\medskip

Let $X$ be a smooth curve (not necessarily complete). First, let us recall
that a chiral algebra over $X$ (see \cite{CHA}, Sect. 3.4 for the detailed 
definition) is a rule that assigns to a natural number $n$ a quasi-coherent
sheaf $\CA_n$ over $X^n$, with a certain factorization data. E.g., 
for $n=2$, we must be given an isomorphism between 
$\CA_2|_{X\times X-\Delta(X)}$ and $\CA_1\boxtimes \CA_1|
_{X\times X-\Delta(X)}$, and an isomorphism
$\CA_2|_{\Delta(X)}\simeq \CA_1$. We must be given a compatible
family of such isomorphisms for any partition $n=n_1+...+n_k$. Each 
$\CA_n$ must be endowed with an equivariant structure with respect to
the symmetric group $\Sigma_n$, and the factorization isomorphisms must
respect this structure. Finally, the collection $\CA_n$ must be endowed
with a unit, which is a map from the collection given by $\CA_{unit}^{(n)}:=
\CO_{X^n}$ to $\CA_n$.

\medskip

The definition of a chiral category is similar. A chiral
category consists of a data of categories $\CC_n$ {\it over} 
\footnote{The notion of abelian category over a scheme or stack 
can be found in \cite{Ga1}. However, for a reasonable definition of chiral 
categories, one needs to work at the triangulated level. We refer the reader
to \cite{Lu} or \cite{FG}, where the corresponding notions have been developed.}
$X^n$ defined for each $n$, endowed with a compatible family of equivalences
such as
$$\CC_2|_{X\times X-\Delta(X)}\simeq \CC_1\boxtimes \CC_1|
_{X\times X-\Delta(X)} \text{ and }
\CC_2|_{\Delta(X)}\simeq \CC_1,$$
etc, and a family of unit objects ${\bf 1}_n\in \CC_n$.

For a given chiral category and a point $x\in X$ we shall denote by
$\CC_x$ the fiber of $\CC_1$ at $x$. In most cases of interest, $\CC_x$
will depend on the formal disc around $x$ in $X$ in a functorial way.
So, one can view the notion of chiral category as that of a category
$\CC$ (thought of as $\CC_x$) endowed with an additional structure.

\medskip

Once the local definition of $\Whit^c(\Gr_G)$ becomes available,
the factorization structure of the affine Grassmannian (see \cite{MV}, Sect. 5)
would imply that $\Whit^c(\Gr_G)$ is a chiral category. With the current
non-local definition, we will consider the corresponding categories 
$\Whit^c(\Gr_G)_n$ separately, without discussing their factorization
properties.

\medskip

Assume for a moment that our ground field is $\BC$, the curve $X$ is 
$\BA^1$, and $\CC$ is a braided monoidal category. In this case one expects to have
a {\it transcendental} procedure that endows $\CC$ with a structure of
chiral category.

\medskip

Thus, the equivalence of \conjref{Lurie} should be understood as an
equivalence of chiral categories, where the chiral structure on the RHS 
is transcendental and comes from the braided monoidal structure.

\ssec{Factorizable sheaves} Let us return to the category $\FS^c(\cG)$,
mentioned before. \footnote{We should note that our definition of
the category $\FS^c(\cG)$ differs from the original one in \cite{FS}
in that we work with twisted D-modules on configuration spaces
rather than with ordinary D-modules. This allows to introduce this category
over an arbitrary curve, whereas in \cite{FS} one was 
restricted to genus $0$.} This category was introduced in a series
of works of M.~Finkelberg and V.~Schechtman in order to upgrade
to the level of an equivalence of categories earlier constructions
of various objects related to $U_q(\cG)$ as cohomology of certain 
sheaves on configuration spaces. 

By its very construction, $\FS^c(\cG)$ is a chiral category. In the present
paper we introduce the corresponding categories $\FS^c(\cG)_n$
explicitly, but we do not discuss their factorization properties, although
the latter are, in a certain sense, evident. 

The main result of \cite{FS}
can be interpreted as follows: we have an equivalence of chiral categories
\begin{equation} \label{FS equiv}
\FS^c(\cG)\simeq \Rep(u_q(\cG)), \,\, q=exp(\pi i c).
\end{equation}
Here $u_q(\cG)$ is the {\it small} quantum group, corresponding to
$\cG$, which coincides with the big quantum group $U_q(\cG)$,
when $q$ is not a root of unity, but is substantially different when it is.
(The latter fact is responsible for our inability to pass from the category
$\Whit^c(\Gr_G)$ to $\Rep(U_q(\cG))$ for all $q$.) 

In \eqref{FS equiv} the RHS acquires a chiral category structure from
the monoidal category structure via the procedure mentioned in
\secref{chiral categories}.

\medskip

The main result of the present paper, \thmref{main}, is that we have an
equivalence for $c\notin \BQ$:
$$\Whit^c(\Gr_G)_n\simeq \FS^c(\cG)_n$$
for every $n$. This should be interpreted as an equivalence of chiral
categories 
\begin{equation} \label{Lurie FS}
\Whit^c(\Gr_G)\simeq \FS^c(\cG),
\end{equation}
thereby proving
\conjref{Lurie} for irrational $c$.

\ssec{Kazhdan-Lusztig equivalence}   \label{KL}

Let us now mention another equivalence of categories, namely that of \cite{KL}.
For a simple group $G_1$ and an invariant form $\kappa_1:\fg_1\otimes \fg_1\to \BC$
consider the category of representations of the corresponding
affine Kac-Moody algebra $\widehat{\fg}_1$,
denoted $\widehat{\fg}_1\mod_{\kappa_1}$. Let us denote by $\KL^{\kappa_1}(G_1)$
the subcategory of $\widehat{\fg}_1\mod_{\kappa_1}$, consisting of representations,
on which the action of the Lie subalgebra $\fg_1[[t]]\subset \widehat{\fg}_1$ integrates
to an action of the group $G_1[[t]]$.

\medskip

Let us write $\kappa_1=\frac{c_1-\check{h}_1}{2\check{h}_1}\cdot
\kappa_{Kil(\fg_1)}$, where $\check{h}_1$ is the dual Coxeter number of $\fg_1$ and
$Kil(\fg_1)$ is the Killing form. Let as assume that $c_1\notin \BQ^{\geq 0}$. 

In \cite{KL} it was shown that $\KL^{\kappa_1}(G_1)$ can be endowed with a structure
of braided monoidal category, and us such it is equivalent to the category
$\Rep(U_{q}(G_1))$, where $q$ and $\kappa_1$ are related by the following formula
$q=exp(\frac{\pi i}{c_1d_1})$ for $c_1$ as above, and where $d_1$ is the ratio
of the squares of lengths of the shortest and longest roots in $\fg_1$.

\medskip

This result is also natural to view in the language of chiral categories. Namely,
both $\widehat{\fg}_1\mod_{\kappa_1}$ and $\KL^{\kappa_1}(G_1)$ 
are naturally chiral categories. We expect that the monoidal structure on 
$\KL^{\kappa_1}(\fg_1)$, defined in \cite{KL}, comes from a chiral
category structure by the procedure mentioned in \secref{chiral categories}.
Thus, the equivalence of \cite{KL}
\begin{equation}  \label{KL equiv} 
\KL^{\kappa_1}(G_1) \simeq \Rep(U_{q}(G_1))
\end{equation}
should be understood in the same framework as \eqref{FS equiv} and \conjref{Lurie}:
it is an equivalence of chiral categories, where the corresponding structure 
on the RHS comes from the braided monoidal structure.

\medskip

Note that we can combine \eqref{FS equiv} and \eqref{KL equiv}, bypassing
the quantum group altogether. We propose:
\begin{conj}  \label{KL FS}
For $c_1\notin \BQ$ we have an equivalence
of chiral categories
$$\KL^{\kappa_1}(G_1) \simeq \FS^{\frac{1}{c_1d_1}}(G_1).$$
This equivalence is algebraic, i.e., exists over an arbitrary ground
field of characteristic $0$.
\end{conj}

Let us note that neither of \eqref{FS equiv}, \eqref{KL equiv} or \conjref{KL FS}
involves Langlands duality. In fact, it appears that \conjref{KL FS} is not so
far-fetched, and is currently the subject of a work in progress.

\ssec{Combining the equivalences}

Let us now combine the discussion in \secref{KL} with \conjref{Lurie}. As will be 
explained in \secref{parameter c}, it is more natural to replace the parameter
$c$ in the definition of $\Whit^c(\Gr_G)$ by an invariant form
$\kappa:\fg\otimes \fg\to \BC$, related to $c$ by
$\kappa=\frac{c-\check{h}}{2\check{h}}\cdot \kappa_{Kil(\fg)}$. 

\medskip

Let $\kappa$ and $\check\kappa$ be invariant forms on $\fg$ and $\cg$,
respectively, related as follows: the forms
$B_{\fh}:=\kappa+\frac{1}{2}\cdot \kappa_{Kil(\fg)}|_{\fh}$ and $B_{\check\fh}:=
\check\kappa+\frac{1}{2}\cdot  \check\kappa_{Kil(\check\fg)}|_{\check\fh}$ 
are non-degenerate and satisfy
$$B_{\check\fh}=B_{\fh}^{-1}.$$
This makes sense, since the Cartan subalgebras $\fh\subset \fg$ and 
$\check\fh\subset \cg$ are mutually dual vector spaces. Assume,
in addition, that the corresponding scalar $\check{c}$ is not in 
$\BQ^{\geq 0}$. (Note that the scalars
$c$ and $\check{c}$ are related by $\check{c}=\frac{1}{cd}$.)

Combining \conjref{KL FS} and \eqref{Lurie FS} we propose:
\begin{conj} \label{Lurie KL}
There exists an equivalence of chiral categories
$$\Whit^c(\Gr_G)\simeq \KL^{\check\kappa}(\cG).$$
\end{conj}

Note that unlike \eqref{Lurie FS} and \conjref{KL FS}, the above \conjref{Lurie FS}
is supposed to hold even for rational (but non-negative) values of $c$.

\ssec{Relation to quantum geometric Langlands correspondence}

We are now going to put the assertion of \conjref{Lurie KL} into the framework
of quantum geometric Langlands correspondence, as was proposed by 
B.~Feigin, E.~Frenkel and A.~Stoyanovsky (see \cite{Sto}), motivated
by an earlier work of Feigin and Frenkel on the duality of W-algebras. 

\medskip

Let $\kappa$ and $\check\kappa$ be as above. For a global curve $X$,
consider the stacks $\Bun_G$ and $\Bun_{\cG}$ of principal bundles
on $X$ with respect to $G$ and $\cG$, respectively,
along with the
corresponding derived categories of twisted D-modules:
$D(\fD\mod{}^c(\Bun_G))$ and $D(\fD\mod{}^{\check{c}}(\Bun_G))$. In 
{\it loc. cit.} the following equivalence was proposed:

\begin{conj} \label{Sto}
$$D(\fD\mod{}{}^{-c}(\Bun_G))\simeq D(\fD\mod{}{}^{\check{c}}(\Bun_{\cG}))$$.
\end{conj}

It is supposed to degenerate to the "usual" geometric Langlands
equivalence
$$D(\fD\mod(\Bun_G))\simeq D(\on{QCoh}(LocSys_{\cG}))$$
as $c\to 0$ and therefore $\check{c}\to \infty$.

\medskip

We would like to propose yet one more conjecture expressing the compatibility
between \conjref{Lurie KL} and \conjref{Sto}. Let us fix points $x_1,...,x_n$
on the curve. We have the categories $\Whit^c(G)_{x_i}$ and $\KL^{\check\kappa}(\cG)_{x_i}$
attached to each of these points.

We claim that there exists a natural functor
$$\on{Poinc}:\Whit^c(G)_{x_1}\times...\times \Whit^c(G)_{x_n}\to 
D(\fD\mod{}^{-c}(\Bun_G)).$$
This is a geometric analog of the Poincar\'e series operator in the theory
of automorphic functions; 
\footnote{The fact that the functor $\on{Poinc}$ flips the sign of $c$
is related to a certain choice we make when we define the category
$\Whit^c(G)$, see \secref{det line}.}
i.e., the adjoint operator to that associating to
an automorphic function its Whittaker model.  
When we interpret the categories 
$\Whit^c(G)_{x_i}$ by the second method adopted in this paper 
(see \secref{Whit trouble}), 
the functor $\on{Poinc}$ corresponds to the direct image by means of the morphism
of stacks $\fW\to \Bun_G$.

\medskip

In addition, there exists a naturally defined functor
$$\on{Loc}:\KL^{\check\kappa}(\cG)_{x_1}\times...\times
\KL^{\check\kappa}(\cG)_{x_n}\to D(\fD\mod{}^{\check{c}}(\Bun_{\cG})).$$
Namely, given objects $V_i\in \KL^{\check\kappa}(\cG)_{x_i}$,
the fiber of $\on{Loc}(V_1,...,V_n)$ at a $\cG$-bundle $\fF_{\cG}$
is given by
$$H_\bullet(\cg^{\fF_{\cG}}_{out},V_1\otimes...\otimes V_n),$$
where $\cg^{\fF_{\cG}}_{out}$ is the Lie algebra of sections 
of the associated bundle with the adjoint representation over
the punctured curve $X-\{x_1,...,x_n\}$.

\medskip

We propose:

\begin{conj}    \label{loc glob}
The diagram of functors
$$
\CD
\Whit^c(G)_{x_1}\times...\times \Whit^c(G)_{x_n}  @>{\text{\conjref{Lurie KL}}}>>
\KL^{\check\kappa}(\cG)_{x_1}\times...\times
\KL^{\check\kappa}(\cG)_{x_n} \\
@V{\on{Poinc}}VV    @V{\on{Loc}}VV    \\
D(\fD\mod{}^{-c}(\Bun_G))  @>{\text{\conjref{Sto}}}>>  
D(\fD\mod{}^{\check{c}}(\Bun_{\cG}))
\endCD
$$
commutes.
\end{conj}

\ssec{Acknowledgments}

Jacob Lurie, whom the main idea of the present paper belongs to,
decided not to sign it in the capacity of author. I would like to express
my gratitude to him for numerous discussions directly and indirectly 
related to the contents of the present paper, as well as a lot of other work 
in progress. 

\medskip

I would also like to express my gratitude to M.~Finkelberg for patient and
generous explanations of the contents of \cite{FS}, which this paper is
based on.

\medskip

The ideas related to quantum geometric Langlands correspondence,
that center around \conjref{loc glob}, have received a crucial impetus 
from communications with A.~Stoyanovsky and from 
discussions with A.~Braverman and E.~Witten.

\medskip

I would also like to thank A.~Beilinson, R.~Bezrukavnikov, E.~Frenkel and 
D.~Kazhdan for useful and inspiring discussions.

\section{Overview}

In this section we shall explain the technical contents of this paper,
section by section.

\ssec{Conventions}

Throughout the paper we work over an arbitrary algebraically closed
ground field of characteristic zero; $X$ will be a smooth projective
curve. We denote by $\omega$ the canonical line bundle on $X$.

\medskip

In the main body of the paper $G$ will be a reductive group with
$[G,G]$ simply connected. We choose a Borel subgroup $B\subset G$,
its opposite $B^-\subset G$ and identify the Cartan quotient
$T:=B/N$ with $B\cap B^-$.

\medskip

By $\cLambda$ we will denote the 
coweight lattice, and by $\Lambda$ its dual--the weight lattice; by $\langle,\rangle$
we will denote the canonical pairing between the two. By 
$\cLambda^{+}$ (resp., $\Lambda^{+}$) we will denote the semi-group
of dominant coweights (resp., weights). By $\Delta$ (resp., $\Delta^+$)
we shall denote the set if roots (resp., positive roots); by $\CI$ we shall
denote the set of vertices of the Dynkin diagram; for $\imath\in \CI$
we let $\alpha_\imath$ (resp., $\check\alpha_\imath$) denote the
corresponding simple root (resp., coroot).
By $\cLambda^{pos}$ (resp., $\Lambda^{pos}$)
we shall denote the positive span of simple co-roots (resp., roots). For $\cla_1,\cla_2\in
\cLambda$ we shall say that $\cla_1\geq \cla_2$ if $\cla_1-\cla_2\in \cLambda^{pos}$.

\medskip

We choose once and for all
a square root $\omegah$ of the canonical bundle $\omega$. For a half-integer
$i$, by $\omega^i$ we will mean $(\omegah)^{\otimes 2i}$. We let
$\omegacrho$ denote the $T$-bundle induced by means of
$2\crho:\BG_m\to T$ from $\omegah$.

\medskip

For  an ind-scheme (or strict ind-stack) by a D-module on it
we shall mean a D-module  supported on some closed
subscheme (or substack).

\ssec{\secref{twisted Whittaker}}  \label{surrogate}

This section is devoted to the surrogate definition of the Whittaker
category $\Whit^c$ using a complete curve $X$. The idea is the following.

Let us think of the field $\BC\ppart$ (resp., the ring $\BC[[t]]$ ) as 
the local field (ring) of a point $x\in X$. Let $N_{out}$ be the 
group-subscheme of $N\ppart$ consisting of maps $(X-x)\to N$.
By construction, the character $\chi:N\ppart\to \BG_a$ is trivial
when restricted to $N_{out}$, so whatever $(N\ppart,\chi)$-equivariant
D-modules are, they should give rise to D-modules on the quotient
$N_{out}\backslash \Gr_G$.

Although $N_{out}\backslash \Gr_G$ makes sense as a functor
on the category of schemes, it is {\it not} a kind of algebraic stack, on which
one can define D-modules directly. However, one can cure this pretty
easily, by embedding it into another object, denoted in this paper
by $\fW_{x}$, the latter being a strict ind-stack and D-modules
on it make sense.

Explicitly, $N_{out}\backslash \Gr_G$ classifies the data of a 
principal $G$-bundle $\fF_G$ over $X$, endowed with a reduction
to $N$ over $X-x$. The stack $\fW_{x}$ replaces the word "reduction"
by "generalized reduction" or "Drinfeld structure". The stacks classifying
these generalized reductions are known as Drinfeld's compactifications;
they are studied in detail, e.g., in \cite{FFKM} or \cite{BG}.

\medskip

Our surrogate Whittaker category $\Whit_x$ is introduced as a category of
D-modules on $\fW_{x}$ that satisfy a certain equivariance condition,
which is supposed to restore the $(N\ppart,\chi)$-equivariance on
$N_{out}\backslash \Gr_G$. Moreover, this equivariance condition
forces objects of $\Whit_x$ to have zero stalks and co-stalks away
from $N_{out}\backslash \Gr_G$. So, our $\Whit_x$ is the "right" 
replacement for $\fD\mod(\Gr_G)^{N\ppart,\chi}$, the only disadvantage
being that the geometric object, on which it is realized, i.e.,  
$N_{out}\backslash \Gr_G$, depends on the choice of the global curve $X$.

The above definition generalizes in a straightforward way to the
case when instead of one point $x$ we have an $n$-tuple $\ox$
of points $x_1,...,x_n$. We obtain a category $\Whit_{\ox}$. One
can prove, but in a somewhat ad hoc way, that the category
$\Whit_{\ox}$ is equivalent to the tensor product 
$\Whit_{x_1}\otimes...\otimes \Whit_{x_n}$. The non-triviality
of the latter comparison is the expression of the non-locality of
our definition of $\Whit_x$.

In the main body of the paper we allow the $n$-tuple $\ox$ vary along
$X^n$; and we will consider the appropriate category of D-modules,
denoted $\Whit_n$. \footnote{Deviating slightly from the notation
pertaining of chiral categories introduced in \secref{chiral categories}, 
our $\Whit_n$ will not be a category over $X^n$, but rather the category 
consisting of objects in the corresponding chiral category, endowed with 
a connection along $X^n$.}

\medskip

The twisted version $\Whit^c_x$ (resp., $\Whit^c_{\ox}$, $\Whit^c_n$) is
defined as follows. By construction, the stack $\fW_x$
(resp., $\fW_{\ox}$, $\fW_n$) is endowed with a forgetful map to
the moduli stack of principal $G$-bundles over $X$, denoted $\Bun_G$.
Instead of ordinary D-modules over $\fW_x$ we consider D-modules,
twisted by means of the $-c$-th power of the pull-back of the determinant line 
bundle over $\Bun_G$. (In the main body of the paper our conventions differ
from those in the introduction, in that for $G$ simple the twisting
parameter $c$ is scaled by $2\check{h}$, i.e., we use as a basic
pairing $\fg\otimes \fg\to \BC$ the Killing form vs. the minimal
integral form.) 

The equivariance condition that singles the Whittaker category among 
all twisted D-modules on $\fW_x$ makes sense, just as it did in the
untwisted case. 

\ssec{\secref{FS cat}}

In this section we review the category of factorizable 
sheaves, introduced in \cite{FS}. 

Let us fix an $n$-tuple of points $\ox$ on our curve $X$. For a coweight $\cmu$,
let $X_{\ox}^\cmu$ denote the space of $\cLambda$-valued divisors on $X$
of total degree $\cmu$,
which are required to be anti-effective away from $\on{supp}(\ox)$. I.e.,
a point of $X_{\ox}^\cmu$ is an expression $\underset{k}\Sigma\, \cmu_k\cdot y_k$,
where $y_k\in X$ are pairwise distinct and $\Sigma\, \cmu_k=\cmu$, and such that
for $y_k\neq x_i$, the corresponding coweight $\cmu_k$ belongs to 
$-\cLambda^{pos}$, i.e., is in the
span of simple coroots with coefficients in $\BZ^{\leq 0}$.

One introduces a line bundle $\CP_{X_{\ox}^\cmu}$ over $X_{\ox}^\cmu$,
which has a local nature, i.e., the fiber of $\CP_{X_{\ox}^\cmu}$ at a point
$\underset{k}\Sigma\, \cmu_k\cdot y_k$ as above is the tensor product of
lines $(\omega^{\frac{1}{2}}_{y_k})^{\otimes N(\cmu_k)}$, 
where $\omega^{\frac{1}{2}}_{y_k}$ denotes the fiber of $\omega^{\frac{1}{2}}$
at $y_k$, and $N(\cmu_k)=(\cmu_k,\cmu_k+2\crho)_{Kil}$. 

\medskip

By definition, a factorizable sheaf consists of a data of a twisted D-module $\CL^\cmu$,
{\it defined for each} $\cmu$, where the twisting is by the $c$-th power of 
$\CP_{X_{\ox}^\cmu}$. The twisted D-modules $\CL^\cmu$
for different $\cmu$ are related by factorization isomorphisms. To explain
what these are it would be easier to pass to the analytic topology. Thus, let
$^1\bU$ and $^2\bU$ be two non-intersecting open subsets of $X$,
such that $\ox\subset {}^1\bU$; let $\cmu=\cmu_1+\cmu_2$. 

We can consider the open subset 
$$^1\bU^{\cmu_1}_{\ox}\times {}^2\bU^{\cmu_2}_\emptyset\subset
X^{\cmu_1}_{\ox}\times X^{\cmu_2}_\emptyset.$$
It admits a natural \'etale map to $X_{\ox}^\cmu$. We require that
the pull-back of $\CL^\cmu$ by means of this map decomposes as a 
product of $\CL^{\cmu_1}$ along the first factor, times a {\it standard}
twisted D-module, denoted $\CL^{\cmu_1}_\emptyset$, along the first
factor.

Thus, the "behavior" of each $\CL^\cmu$ near a point $\underset{k}\Sigma\, 
\cmu_k\cdot y_k$ is the tensor product over $\cmu_k$ of the "behaviors"
of $\CL^{\cmu_k}$ near $\cmu_k\cdot y_k\in X^{\cmu_k}_{\ox}$, and the
latter is pre-determined, unless $y_k$ coincides with one of the $x_i$'s.

\medskip

In this way one obtains a category that we denote by $\wt{\FS}{}^c_\ox$. 
One singles out the desired subcategory $\FS^c_\ox\subset \wt{\FS}{}^c_\ox$
by imposing a condition on the singular support.

\medskip

In \cite{FS} it is shown that the category $\FS^c_\ox$ of factorizable sheaves (when
$\ox=\{x\}$) is equivalent to that of representations of the small quantum
group $u_q(\cG)$, where the lattice $\cLambda$, being the set of coweights
of $G$, plays the role of the weight lattice for $\cG$.

\ssec{\secref{Zast}} 

In this section we show how to construct a functor from $\Whit^c_\ox$ to $\FS^c_\ox$,
i.e., how to pass from twisted D-modules on stacks $\fW_\ox$ to twisted
D-modules on configurations spaces $X^\cmu_\ox$, $\cmu\in \cLambda$.

The idea goes back to the construction of the Satake homomorphism
from the spherical Hecke algebra to the algebra of functions on the 
Cartan subgroup of $\cG$, and whose geometric analog was used in 
\cite{MV} to construct the fiber functor
in the usual case of $\fD\mod(\Gr_G)^{G[[t]]}$. Namely, the $\cla$-weight
space is recovered as cohomology of the orbit of the group $N\ppart$ (or
$N^-\ppart$, as the two are $G[[t]]$-conjugate) passing through the point 
$t^\cla\in \Gr_G$. 

We would like to do the same for $\Whit^c=\fD\mod^c(\Gr_G)^{N\ppart,\chi}$,
namely take cohomology along the orbits of $N^-\ppart$. The latter makes
sense (at least non-canonically), as the line bundle $\CP_{\Gr_G}$ trivializes
over $N^-\ppart$-orbits.

However, in our context, for each $\cla$ we want to obtain not just one vector
space, but a twisted D-module over $X_x^\cla$, with the factorization property.
The idea is to repeat the above procedure of taking cohomology along
$N^-\ppart$-orbits in the family parameterized by $X_x^\cla$. Such families
for all $\cla$ are provided by Zastava spaces, and they are well-adapted
for our definition of $\Whit^c$.

\medskip

Zastava spaces were introduced in \cite{FFKM}. By definition, a point of the
Zastava space $\CZ^\cmu_\ox$ is
a point of $\fW_\ox$, i.e., a $G$-bundle, endowed with a generalized reduction
to $N$ away from $\on{supp}(\ox)$, and additionally, with a reduction to
$B^-$ of degree $\cmu$ (up to a shift by $(2g-2)\cdot \crho$), defined everywhere,
which at the generic point of $X$ is transversal to the given
reduction to $N$. The measure of global non-transversality of the two reductions
is given by a point of $X^\cmu_\ox$, i.e., we have a map $\pi:\CZ^\cmu_\ox\to
X^\cmu_\ox$. 

It is a basic observation of \cite{FFKM} (see also \cite{BFGM}) that the space
$\CZ^\cmu_\ox$ {\it factorizes} over $X^\cmu_\ox$. Namely, for a point
$\cmu_k\cdot y_k\in X^{\cmu_k}_{\ox}$, the fiber of $\pi$ over it is a product
$\underset{k}\Pi\, \CZ^{\cmu_k}_{loc,y_k}$, where each 
$\CZ^{\cmu_k}_{loc,y_k}$ is a subscheme of the affine Grassmannian
$\Gr_G=G\ppart/G[[t]]$ (here $t$ is a local parameter at $y_k$)
that depends only on $\cmu_k$.

\medskip

Another crucial observation is that the line bundle on $\CZ^\cmu_\ox$, 
equal to the pull-back (from $\fW_\ox$ of the pull back) of the determinant
line bundle on $\Bun_G$, is canonically the same as the pull-back by
means of $\pi$ of the line bundle $\CP_{X_{\ox}^\cmu}$ over $X_{\ox}^\cmu$.
Therefore, we have a direct image functor between the corresponding categories
of twisted D-modules.

This allows us to construct the desired functor $\sG:
\Whit^c_\ox\to \FS^c_\ox$. Namely, starting from $\CF\in \Whit^c_\ox$,
we let the $\cmu$-component $\CL^\cmu$ of $\sG(\CF)$ to be the
direct image under $\pi$ of the pull-back of $\CF$ from $\fW_\ox$
to $\CZ^\cmu_\ox$. In order to show that $\CF\mapsto \{\CL^\cmu\}$
defined in this way is indeed a (reasonable) functor $\Whit^c_\ox\to \FS^c_\ox$
we need to check several things.

\medskip

First, in \propref{factor of Whit},
we establish a factorization property of twisted D-modules on $\CZ^\cmu_\ox$
obtained by pull-back from twisted D-modules on $\fW_\ox$ that belong to
$\Whit^c_\ox$. For an object $\CF\in \Whit^c_\ox$, its restriction to the
fiber of $\pi$ over a point of $X^\cmu_\ox$, identified by the above with
a product $\underset{k}\Pi\, \CZ^{\cmu_k}_{loc,y_k}$, decomposes as an
external product $\underset{k}\boxtimes\,\CF^{\cmu_k}_{loc,y_k}$.

This insures, among the rest, that although the forgetful map
$\CZ^\cmu_\ox\to \fW_\ox$ is not in general smooth, the pull-back functor
applied to twisted D-modules that belong to $\Whit^c$ does not produce higher
or lower cohomologies.

\medskip

The next step consists of analyzing the direct image with respect to the morphism
$\pi:\CZ^\cmu_\ox\to X^\cmu_\ox$. \propref{perversity} insures that this
operation does not produce higher or lower cohomology either. Moreover,
\thmref{cleanness} states that when $c\notin \BQ$,
the direct image coincides with the direct image with compact supports.

Finally, we need to establish that the system
$\{\CL^\cmu\}$ satisfies the required factorization property. The fact that
it satisfies {\it some} factorization property (i.e., one, where the standard 
twisted D-module 
$\CL^{\cmu}_\emptyset$ is replaced by a certain $'\CL^{\cmu}_\emptyset$)
is an immediate corollary of \propref{factor of Whit}, mentioned above. The
fact that $'\CL^{\cmu}_\emptyset\simeq \CL^{\cmu}_\emptyset$ is one of the
main technical points of the present paper and is the content of \thmref{F0}.

\ssec{Sections \ref{proofs A} and \ref{sect F0}}

These two sections are devoted to the proofs of \propref{perversity}
and \thmref{F0}.

The proof of \propref{perversity} has two ingredients. One is the
estimate on the dimension of the fibers of the map
$\pi:\CZ^\cmu_\ox\to X^\cmu_\ox$, which amounts to estimating
the dimensions of the schemes $\CZ^{\cmu}_{loc}$. The second
step consists of showing that 
the lowest degree cohomology in
\begin{equation} \label{c est}
H(\CZ^\cmu_{loc,y},(\CF_\emptyset)^\cmu_{loc}) 
\end{equation}
vanishes, where
$\CF_\emptyset$ is the "basic" object of $\Whit$. This amounts to
a calculation that follows from \cite{FGV}, Proposition 7.1.7 coupled with
\cite{BFGM}, Prop. 6.4.

\medskip

\thmref{F0} is concerned with the direct image under $\pi$ of the 
pull-back of the basic object $\CF_\emptyset$ to 
$\CZ^\cmu_\emptyset$ (we denote the resulting twisted D-module
on $X^\cmu_\emptyset$ by $'\CL^\cmu_\emptyset$). We have to
identify it with the standard twisted D-module $\CL^\cmu_\emptyset$.

Part (1) of the theorem asserts that this identification exists 
away from the diagonal divisor on 
$X^\cmu_\emptyset$. This amounts to identifying the cohomology
$H(\BG_m,\chi\otimes \Psi(c))$ with $\BC$, where $\chi$ is the Artin-Schreier
D-module on $\BG_a$, and $\Psi(c)$ is the Kummer D-module on 
$\BG_m$, corresponding to the scalar $c$.

Part (2) of the theorem asserts that $'\CL^\cmu_\emptyset$ is the
Goresky-MacPherson extension of its restriction to the complement of the
diagonal divisor. This is true only under the assumption that $c$
is irrational. The latter amounts to two things: one is the essential
self-duality of $'\CL^\cmu_\emptyset$, which follows from \thmref{cleanness};
the other is the assertion that for $\cmu$, not equal to the negative of
one of the simple co-roots, the cohomology \eqref{c est} vanishes
also in the sub-minimal degree, allowed by dimension considerations.
This is done by a direct analysis.

\ssec{\secref{proof of cleanness}} In this section we prove
\thmref{cleanness}, which states that the functor $\sG:\Whit_n^c\to \FS^c_n$
essentially commutes with Verdier duality. More precisely, we prove
that for a twisted D-module on $\CZ^\cmu_n$, obtained as a
pull-back of an object $\CF\in \Whit^c_n\subset \fD\mod^c(\fW_n)$,
its direct image onto $X^\cmu_n$ under $\pi$ equals the direct
image with compact supports.

\medskip
 
To prove this fact we introduce a compactification $\BunBmb^\cmu$
of the stack $\Bun_{B^-}^\cmu$ along the fibers of the projection
$\Bun_{B^-}^\cmu\to \Bun_G$, by allowing the reduction of a $G$-bundle
to $B^-$ to degenerate to a Drinfeld structure. 

Whereas $\CZ^\cmu_n$ was an open sub-stack of the fiber product
$\fW_n\underset{\Bun_G}\times \Bun_{B^-}^\cmu$, we define 
$\ol\CZ^\cmu_n$ to be the corresponding open sub-stack of
$\fW_n\underset{\Bun_G}\times \BunBmb^\cmu$. The map
$\pi:\CZ^\cmu_n\to X^\cmu_n$ extends to a map 
$\ol\pi:\ol\CZ^\cmu_n\to X^\cmu_n$. The main observation
is that the map $\ol\pi$ is proper.

\medskip

Let $\jmath^-$ denote the open embedding $\Bun_{B^-}^\cmu\hookrightarrow
\BunBmb^\cmu$, and denote by $'\jmath^-$ the base-changed map
$\CZ^\cmu_n\hookrightarrow \ol\CZ^\cmu_n$. \thmref{cleanness} is an
easy corollary of another result, \thmref{thm cleanness}, that states
that for $\CF\in \Whit^c_n$, its pull-back to $\CZ^\cmu_n$ is {\it clean} with
respect to $'\jmath^-$, i.e., the direct image $'\jmath^-_*$ equals 
$'\jmath^-_!$.

\medskip

We deduce \thmref{thm cleanness} from \thmref{other cleanness} that
states the "constant" twisted D-module on $\Bun_{B^-}^\cmu$ is clean
with respect to $\BunBmb^\cmu$. The latter theorem is proved by a 
word-for-word repetition of the calculation of the intersection cohomology
sheaf on $\BunBmb^\cmu$, performed in \cite{BFGM}.

\ssec{\secref{equivalence}} In this section we finish the proof
of the fact that the functor $\sG:\Whit^c_n\to \FS^c_n$ is an 
equivalence.

We first show that the functor $\sG$ induces an equivalence
for a fixed set of pole points $\ox=x_1,...,x_n$:
\begin{equation} \label{open equiv}
\Whit^c_\ox\to \FS^c_\ox.
\end{equation}

This essentially reduces to the fact that the category $\Whit^c_\ox$
is semi-simple for $c$ irrational. 

\medskip

The final step is to show that the equivalences \eqref{open equiv} 
glue together as $\ox$ moves along $X^n$. The essential ingredient
here is \thmref{cleanness} that asserts that the functor $\sG$ is
essentially Verdier self-dual.

\section{The twisted Whittaker category}   \label{twisted Whittaker}

\ssec{}

For $n\in \BZ^{\geq 0}$ we consider the $n$-th power of $X$, and
the corresponding version of the Drinfeld compactification, denoted
$\fW_n$ over $X^n$. By definition, $\fW_n$ is the ind-stack that
classifies the following data:

\begin{itemize}
 
\item An $n$-tuple of points $x_1,...,x_n$ of $X$.

\item A $G$-bundle $\fF_G$ on $X$ (For a dominant weight
$\lambda$ we shall denote by $\CV^\lambda_{\fF_G}$ the vector
bundle associated with the corresponding highest weight
representation.)

\item For each dominant weight $\lambda$ a non-zero map
$$\kappa^\lambda:\omega^{\langle \lambda,\crho\rangle}\to \CV^\lambda_{\fF_G},$$
which is allowed to have poles at $x_1,...,x_n$. The maps $\kappa^\lambda$
are required to satisfy the Pl\"ucker relations (see \cite{BG}, Sect. 1.2.1).

\end{itemize}

Let $\fp$ denote the natural forgetful map $\fW_n\to \Bun_G$. When $n=0$ we shall
use the notation $\fW_\emptyset$, or sometimes simply $\fW$.

\ssec{}    \label{det line}

Let $\CP_{\Bun_G}$ be the determinant line bundle on $\Bun_G$.
We normalize it so that the fiber over $\fF_G\in \Bun_G$
is 
$$\det R\Gamma(X,\fg_{\fF_G})\otimes \left(\underset{\alpha\in \Delta}\bigotimes \det 
R\Gamma\left(X,\omega^{\langle \alpha,\crho\rangle}\right)\right)^{\otimes -1}
\bigotimes
\biggl(\det R\Gamma\left(X,\CO\right)\biggr)^{-\dim({\mathfrak t})},$$
where ${\mathfrak t}$ is the Cartan subalgebra of $G$.
(The second and third factors are lines that do not depend on the point of
$\Bun_G$; the reason
for introducing them will become clear later.)

\medskip

Let $\CP_{\fW_n}$ denote the {\it inverse} of the 
pull-back of $\CP_{\Bun_G}$  to $\fW_n$ by means of
$\fp$. For a scalar we shall denote by $\fD\mod^c(\fW_n)$ the category
of $"(\CP_{\fW_n})^{\otimes c}"$-twisted D-modules on $\fW_n$.  
When $c=0$ (or, more generally, when $c$ is an integer) this category is 
canonically equivalent to $\fD\mod(\fW_n)$.

\ssec{}  \label{good at y}

We are now going to introduce a full subcategory of $\fD^c\mod(\fW_n)$,
denoted $\Whit_n^c$. When $c=0$, this is the Whittaker category of
\cite{FGV}; for an arbitrary $c$ the definition is not much different.
As the definition follows closely \cite{FGV}, Sect. 6.2 and \cite{Ga}, Sect. 4
we shall omit most of the proofs and the refer the reader to {\it loc. cit.}.

\medskip

Fix a point $y\in X$ let $\CB_{y}^{\reg}$ (resp., $\CB_{y}^{\mer}$)
be the group (resp., group ind-scheme) of automorphisms of
the $B$-bundle induced by means of $T\to B$ from $\omegacrho$
over the formal disc $\D_y$ (resp., formal punctured disc $\D^\times_y$) around $y$. 
This group is non-canonically isomorphic to $B(\CO_{y})$
(resp., $B(\CK_{y})$), where $\CO_{y}$ (resp., $\CK_{y}$)
is the completed local ring (resp., field) at $y$.

Let $\CN_{y}^{\reg}\subset \CB_{y}^{\reg}$ (resp., $\CN_{y}^{\mer}\subset \CB_{y}^{\mer}$)
be the kernel of the natural homomorphism $\CB_{y}^\reg\to T(\CO_{y})$
(resp., $\CB_{y}^\mer\to T(\CK_{y})$). Note that 
$$\CN_{y}^{\mer}/[\CN_{y}^{\mer},\CN_{y}^{\mer}]\simeq 
\underset{r\text{ times}}{\underbrace{\omega|_{\D^\times_y}\times...\times \omega|_{\D^\times_y}}},$$
where $r$ is the semi-simple rank of $G$. Taking the residue along each component
we obtain a canonical homomorphism $\chi_y:\CN_{y}^{\mer}\to \BG_a$.

\medskip

Fix a non-empty collection of distinct points $\oy:=y_1,...,y_m$, and set 
$\CN^\reg_{\oy}$ (resp., $\CN^\mer_{\oy}$) to be the product of the corresponding groups
$\CN_{y_j}^{\reg}$ (resp., $\CN_{y_j}^{\mer}$). We shall denote by
$\chi_{\oy}$ the corresponding homomorphism $\CN^\mer_{\oy}\to \BG_a$.

Consider an open substack $\left(\fW_n\right)_{\text{good at }\oy}$
of $\fW_n$ corresponding to the condition that the points $x_1,...,x_n$ 
stay away from $\oy$, and the maps $\kappa^\lambda$ are injective
on the fibers over $y_j$, $j=1,...,m$ (this is equivalent to asking that $\kappa^\lambda$
be an injective bundle map on a neighborhood of these points). 

Note that a point of this substack defines a $B$-bundle over each $\D_{y_j}$, such that the induced $T$-bundle is $\omegacrho|_{\D_{y_j}}$.
We define a $\CN^\reg_{\oy}$-torsor over $\left(\fW_n\right)_{\text{good at }\oy}$ 
that classifies the data as above plus an additional choice of 
identification $\beta_{y_j}$ of this $B$-bundle with 
$B\overset{T}\times \omegacrho|_{\D_{y_j}}$,
which is compatible with the existing identification of the corresponding $T$-bundles. 

Let us denote the resulting stack by $_{\oy}\fW_n$. The standard re-gluing construction
equips $_{\oy}\fW_n$ with an action of the group ind-scheme $\CN^\mer_{\oy}$
(see \cite{FGV}, Sect. 3.2 or \cite{Ga}, Sect. 4.3). 

\medskip

Let $\CP_{_{\oy}\fW_n}$ be the pull-back of the line bundle
$\CP_{\fW_n}$ to $_{\oy}\fW_n$.

\begin{lem}
The action of
$\CN^\mer_{\oy}$ on $_{\oy}\fW_n$ naturally lifts to an action on 
$\CP_{_{\oy}\fW_n}$.
\end{lem}

\begin{proof}
For a point
$$\{(x_1,...,x_n),\fF_G,\kappa^\lambda,\beta_j\}\in {}_{\oy}\fW_n \text{ and }
\{\bn_j\in \CN_{y_j}^{\mer}\},$$ let $\{(x_1,...,x_n),\CF'_G,\kappa'{}^\lambda,\beta'_j\}$
be the corresponding new point of $_{\oy}\fW_n$. We have to show that the lines
$$\left(\CP_{\Bun_G}\right)_{\fF_G} \text{ and }\left(\CP_{\Bun_G}\right)_{\CF'_G}$$ are
canonically isomorphic. 

However, the ratio of these two lines is canonically
isomorphic to the product over $j=1,...,m$ of relative determinants of the 
$G$-bundles $\fF_G|_{\D_{y_j}}$ and $\CF'_G|_{\D_{y_j}}$, which by definition
are identified over the corresponding punctured discs $\D^\times_{y_j}$.
Both these bundles are equipped with reductions to $B$ that coincide
over $\D^\times_{y_j}$ and such that the induced isomorphism of $T$-bundles
is regular over the non-punctured disc. This establishes the required isomorphism
between the lines.
\end{proof}

\ssec{}   \label{def cat}

We define $\left(\Whit_n^c\right)_{\text{good at }\oy}$ to be
the full subcategory of $\fD\mod^c\left(\left(\fW_n\right)_{\text{good at }\oy}\right)$
consisting of D-modules on $_{\oy}\fW_n$ that are 
$(\CN^\mer_{\oy},\chi_{\oy})$-equivariant   
\footnote{The role of the Artin-Schreier sheaf in the world of D-modules in plaid 
by the exponential D-module $exp$ on $\BG_a$. We recall that
the D-module $exp$ on $\BG_a$ is generated by one section $"e^z"$
that satisfies the relation $\partial_z\cdot "e^z"="e^z"$, where $z$ is
a coordinate on $\BG_a$.} (see \cite{FGV}, Sect. 6.2.6 or
\cite{Ga}, Sect. 4.7).

Note that the $\CN^\reg_{\oy}$-equivariance
condition canonically descends any such D-module from $_{\oy}\fW_n$
to $\left(\fW_n\right)_{\text{good at }\oy}$.

\medskip

Let now $\oy'$ and $\oy''$ be two collections of points, and set $\oy=\oy'\cup \oy''$.
Note that 
$$\left(\fW_n\right)_{\text{good at }\oy'}\cap \left(\fW_n\right)_{\text{good at }\oy''}=
\left(\fW_n\right)_{\text{good at }\oy}.$$

In particular, we can consider the corresponding groups $\CN^\reg_{\oy'}$,
$\CN^\reg_{\oy''}$ and $\CN^\reg_{\oy}$ and torsors with respect to them over 
$\left(\fW_n\right)_{\text{good at }\oy}$. Let us consider the corresponding three
subcategories of $\fD\mod^c\left(\left(\fW_n\right)_{\text{good at }\oy}\right)$.
As in \cite{Ga}, Corollary 4.14 one shows that, as long as $\oy'$ and $\oy''$ 
are non-empty, the above three subcategories coincide. 

\medskip

This shows that we have a well-defined 
full-subcategory $\Whit_n^c\subset \fD^c\mod(\fW_n)$: an object $\CF$ belongs
to $\Whit_n^c$ if for any $\oy$ as above,
its restriction to any $\left(\fW_n\right)_{\text{good at }\oy}$
belongs to $\left(\Whit_n^c\right)_{\text{good at }\oy}$.

\ssec{}  \label{Whit at point}

Let us fix points $\ox:=x_1,.,,x_n$ and denote by $\fW_{\ox}$ 
the fiber over the corresponding point of $X^n$.
With no restriction of generality we can assume that all the points
$x_i$ are distinct. Let $\Whit^c_\ox$ denote the corresponding category
of twisted D-modules on $\fW_{\ox}$.

The same analysis as in \cite{FGV}, Lemma 6.2.4 or \cite{Ga}, Prop. 4.14,
shows that every object of $\Whit_\ox^c$
is holonomic, and one obtains the following explicit description of the
irreducibles (and some other standard objects) in this category.

Let $\ol{\cla}=\cla_1,...,\cla_n$ be an $n$-tuple of dominant coweights
of $G$. Let $\fW_{\ox,\ol{\cla}}$ be a locally closed substack of
$\fW_{\ox}$ consisting of points $\{\fF_G,\kappa^\lambda\}$,
where each $\kappa^\lambda$ has a pole of order 
$\langle \lambda,\cla_i\rangle$ at $x_i$ and no zeroes anywhere else.
Let $\jmath_{\ox,\ol{\cla}}$ denote the corresponding locally closed
embedding; by \cite{FGV}, Prop. 3.3.1, this map is affine. 

\medskip

Proceeding as above, for every such $\ol{\cla}$, one can introduce the
category $\Whit_{\ox,\ol{\cla}}^c$. The following is shown as 
\cite{FGV}, Lemma 6.2.4 or \cite{Ga}, Prop. 4.13:

\begin{lem}  \label{on stratum}
The category $\Whit_{\ox,\ol{\cla}}^c$ is (non-canonically) equivalent
to that of vector spaces.
\end{lem}

Let $\CF_{\ox,\ol{\cla}}$ denote the unique irreducible object of the above
category (it is a priori defined up to a non-canonical scalar automorphism). 
Let $\CF_{\ox,\ol{\cla},!}$
(resp., $\CF_{\ox,\ol{\cla},*}$, $\CF_{\ox,\ol{\cla},!*}$) denote its extension
by means of $\jmath_{\ox,\ol{\cla},!}$ (resp., $\jmath_{\ox,\ol{\cla},*}$,
$\jmath_{\ox,\ol{\cla},!*}$) on the entire $\fW_{\ox}$. All of the above
objects are D-modules since the map $\jmath_{\ox,\ol{\cla}}$ is affine
(see \cite{FGV}, Prop. 3.3.1 or \cite{BFG}, Theorem 11.6 for 
an alternative proof). As in \cite{FGV}, Prop. 6.2.1
or \cite{Ga}, Lemma 4.11, one shows that all three are objects of 
$\Whit^c_\ox$.

\begin{lem}  \hfill    \label{descr of irr in Whit}

\smallskip

\noindent{\em(a)} Every irreducible in $\Whit^c_\ox$ is of the form 
$\CF_{\ox,\ol{\cla},!*}$ for some $\ol{\cla}$.

\smallskip

\noindent{\em(b)} 
The cones of the canonical maps
\begin{equation} \label{non-cleanness}
\CF_{\ox,\ol{\cla},!}\to \CF_{\ox,\ol{\cla},!*}\to \CF_{\ox,\ol{\cla},*}
\end{equation}
are extensions of objects $\CF_{\ox,\ol{\cla}{}',!*}$
for $\ol\cla{}'<\ol\cla$.
\end{lem}

It is a basic fact (which is the main theorem of \cite{FGV}) 
that for $c=0$ the canonical maps in \eqref{non-cleanness}
are isomorphisms. This will no longer be true for an arbitrary $c$,
but as we shall show, it will still be true for $c\notin \BQ$.

\ssec{}

Let us describe more explicitly the basic object of the category $\Whit^c_\emptyset$,
which we shall denote by $\CF_\emptyset$. Consider the open substack
$\fW_{\emptyset,0}\subset \fW_\emptyset$. From the definition of the line bundle
$\CP_{\Bun_G}$ we obtain:

\begin{lem}
The restriction of $\CP_{\fW_\emptyset}$ to $\fW_{\emptyset,0}$ admits
a canonical trivialization.
\end{lem}

Thus, the category $\Whit^c_{\emptyset,0}$ is the same as $\Whit_{\emptyset,0}$.

\medskip

In addition, the sum of residues gives rise to a map $\fW_{\emptyset,0}\to \BG_a$.
We define $\CF_{\emptyset,0}$ as the object of $\Whit^c_{\emptyset,0}$,
corresponding to the pull-back of the D-module $exp$ from
$\BG_a$ to $\fW_{\emptyset,0}$ under this morphism. 

Since there are no dominant weights $\leq 0$, from \lemref{descr of irr in Whit}(b)
we obtain:
$$\jmath_{\emptyset,0,!}(\CF_{\emptyset,0})\simeq
\jmath_{\emptyset,0,!*}(\CF_{\emptyset,0})\simeq \jmath_{\emptyset,0,*}(\CF_{\emptyset,0}).$$
We set $\CF_\emptyset$ to be the above object of $\Whit^c_\emptyset$.

\ssec{The parameter "$c$"}   \label{parameter c}

Note that when the adjoint group, corresponding to $G$, is semi-simple
(and not simple), the line bundle $\CP_{\Bun_G}$ is naturally a product
of lines bundles corresponding to the simple factors of $G_{ad}$. Therefore,
when defining the categories $\Whit^c_n$, instead of one scalar
$c$ one can work with a $k$-tuple of scalars, where $k$ is the number of
simple factors. 

More invariantly, from now on we shall understand $c$
as an ad-invariant symmetric bilinear form $\fg_{ad}\otimes \fg_{ad}\to \BC$, or 
equivalently, a Weyl group invariant symmetric bilinear form 
$(,)_c:\cLambda_{G_{ad}}\otimes \cLambda_{G_{ad}}\to \BC$.

\medskip

We will say that $c$ is "integral" if the latter form takes integral values.
In this case it is known that $"(\CP_{\Bun_G})^{\otimes c}"$ is defined
as a line bundle. Hence, the categories $\Whit^{c'}_n$ and $\Whit^{c''}_n$
with $c''-c'$ integral are equivalent.

\medskip

We will say that $c$ is non-integral if 
$(\check\alpha_\imath,\check\alpha_\imath)_c\notin \BZ$ for
any $\imath\in \CI$. Note, that unless $G$ is simple, not "integral"
is not the same as "non-integral". 

\medskip

We shall say that $c$ is irrational if the restriction of $c$ to each of the
simple factors is an irrational multiple of the Killing form. Equivalently,
this means that $(\check\alpha_\imath,\check\alpha_\imath)_c\notin \BQ$
for any $\imath\in \CI$.

\section{The FS category}   \label{FS cat}

From now on in the paper we will assume that $c$ is non-integral.

\ssec{}

For $n\in \BZ^{\geq 0}$ and $\cmu\in \cLambda$ we introduce
the ind-scheme $X^\cmu_n$, fibered over $X^n$ to classify
pairs $\{(x_1,...,x_n)\in X^n,D\}$, where $D$ is a $\cLambda$-valued divisor 
on $X$ of total degree $\cmu$ with the condition that for every
dominant weight $\lambda$, the $\BZ$-valued divisor $\langle \lambda,D\rangle$ is
anti-effective away from $x_1,...,x_n$.

\medskip

When $n=0$, we shall use the notation $X^\cmu_\emptyset$, or sometimes simply
$X^\cmu$. For this scheme to be non-empty we need that 
$\cmu\in -\cLambda^{pos}$. If $\cmu=-\Sigma\, m_\imath\cdot \check\alpha_\imath$,
we have
\begin{equation} \label{sym power}
X^\cmu_\emptyset=\underset{\imath}\Pi\, X^{(m_\imath)}.
\end{equation}

\medskip

We can represent $X^\cmu_n$ explicitly as a union of schemes as follows.
Fix an $n$-tuple $\ol{\cla}=\cla_1,...,\cla_n$ of elements of $\cLambda$.
We define a closed subscheme $X^\cmu_{n,\leq \ol\cla}\subset X^\cmu_n$
by the condition that the divisor $D':=D-\underset{i=1,...,n}\Sigma\, \cla_i\cdot x_i$
is such that $\langle \lambda,D'\rangle$ is anti-effective. By adding the
divisor $\underset{i=1,...,n}\Sigma\, \cla_i\cdot x_i$, we identify the scheme
$X^\cmu_{n,\leq \ol\cla}$ with $X^{\cmu-\cla_1-...-\cla_n}_\emptyset$.

For another $n$-tuple $\ol{\cla}'=\cla'_1,...,\cla'_n$ with $\cla'_i\geq \cla_i$
we have a natural closed embedding $X^\cmu_{n,\leq \ol\cla}\hookrightarrow
X^\cmu_{n,\leq \ol\cla'}$. It is clear that 
$$X^\cmu_n=\underset{\ol\cla}{\underset{\longrightarrow}{lim}}\, X^\cmu
_{n,\leq \ol\cla}.$$


\ssec{}   \label{intr line bundle}

We are now going to introduce a certain canonical line bundle $\CP_{X^\cmu_n}$
over $X^\cmu_n$. Consider the stack $\Bun_T\simeq 
\Pic(X)\underset{\BZ}\otimes \cLambda$. On it we consider the following
line bundle $\fF_{\Bun_T}$: its fiber at $\CP_T\in \Bun_T$ is
the line
$$\left(\underset{\alpha\in \Delta}\bigotimes \det 
R\Gamma\left(X,\alpha(\fF_T)\right)\right)\otimes
\left(\underset{\alpha\in \Delta}\bigotimes \det 
R\Gamma\left(X,\omega^{\langle \alpha,\crho\rangle}\right)\right)^{\otimes -1}.$$

Consider the Abel-Jacobi map
$$AJ:X^\cmu_n\to \Bun_T$$ that sends
a point $\{(x_1,...,x_n)\in X^n,D\}$ to the $T$-bundle $\omega^{\crho}(-D)$.
We set 
$$\CP_{X^\cmu_n}:=AJ^*(\CP^{\otimes -1}_{\Bun_T}).$$

\medskip

The main property of the line bundle $\CP_{X^\cmu_n}$ is that it has
a local nature in $X$:

\begin{lem}  \label{fiber of line bundle}  
%
%
%
The fiber of $\CP_{X^\cmu_n}$ at point $\{(x_1,...,x_n)\in X^n,D\}\in X^\cmu_n$ 
with $D=\Sigma\, \cmu_k\cdot y_k$,
is canonically isomorphic to 
$$\underset{k}\bigotimes\, (\omega^{\frac{1}{2}}_{y_k})^{\otimes (\cmu_k,\cmu_k+2\crho)_{Kil}},$$
where $(\cdot,\cdot)_{Kil}$ is the Killing form on $\cLambda$.
%
%
\end{lem}

This lemma implies in particular that the line bundle $\CP_{X^\cmu_n}$
can be defined over $X^\cmu_n$ for a not necessarily complete curve $X$.
\footnote{The construction of the line bundle $\CP_{X^\cmu_n}$ starting
from the form $(\cdot,\cdot)_{Kil}$ is a particular case of a $\theta$-data,
see \cite{CHA}, Sect. 3.10.3.}

\medskip

For a partition $n=n_1+n_2$, $\cmu=\cmu_1+\cmu_2$, consider the natural
addition map $\on{add}_{\cmu_1,\cmu_2}: 
X^{\cmu_1}_{n_1}\times X^{\cmu_1}_{n_2}\to X^\cmu_n$.

For an $n_1$-tuple and an $n_2$-tuple of elements of $\cLambda$, $\ol\cla_1$
and $\ol\cla_2$, respectively, let $X^{\cmu_1}_{n_1,\leq\ol\cla_1}\subset
X^{\cmu_1}_{n_1}$ and $X^{\cmu_2}_{n_2,\leq\ol\cla_2}\subset
X^{\cmu_2}_{n_2}$ be the corresponding closed subschemes.
Let
$$\left(X^{\cmu_1}_{n_1,\leq\ol\cla_1}\times X^{\cmu_1}_{n_2,\leq\ol\cla_2}\right)_{disj}$$
be the open part of the product, corresponding to the condition that the
supports of the corresponding divisors are disjoint. Note that the restriction
of the map $\on{add}_{\cmu_1,\cmu_2}$ to this open subset
is an \'etale map to the subscheme $X^{\cmu}_{n,\leq\ol\cla_1\cup\ol\cla_2}$.

\medskip

From \lemref{fiber of line bundle} we obtain the following 
{\it factorization property}:
\begin{equation} \label{line bundle factorization}
\on{add}_{\cmu_1,\cmu_2}^*(\CP_{X^\cmu_n})
|_{\left(X^{\cmu_1}_{n_1,\leq\ol\cla_1}\times X^{\cmu_1}_{n_2,\leq\ol\cla_2}\right)_{disj}}
\simeq
\CP_{X^{\cmu_1}_{n_1}}\boxtimes \CP_{X^{\cmu_2}_{n_2}}|_
{\left(X^{\cmu_1}_{n_1,\leq\ol\cla_1}\times X^{\cmu_1}_{n_2,\leq\ol\cla_2}\right)_{disj}},
\end{equation}
compatible with refinements of partitions.

\medskip

Let us also denote by $(X^{\cmu_1}_\emptyset\times X^{\cmu_2}_n)_{disj}$
ind-subscheme of $X^{\cmu_1}_\emptyset\times X^{\cmu_2}_n$
consisting of points $$\{D_1\in X^{\cmu_1}_\emptyset,\ox\in X^n,
D_2\in X^{\cmu_2}_n\},$$ such that $D_1$ is disjoint from both $\ox$ and $D_2$.
Let us denote by $\on{add}_{\cmu_1,\cmu_2,disj}$ the restriction of the map
$\on{add}_{\cmu_1,\cmu_2}$ to this open subset; it is also \'etale. We have
an isomorphism
\begin{equation} \label{basic line bundle factorization}
\on{add}_{\cmu_1,\cmu_2,disj}^*(\CP_{X^\cmu_n})
\simeq
\CP_{X^{\cmu_1}_\emptyset}\boxtimes \CP_{X^{\cmu_2}_{n}}.
\end{equation}

\ssec{}

Let $\fD\mod^c(X^\cmu_n)$ denote the category of 
$"\CP_{X^\cmu_n}^{\otimes c}"$-twisted D-modules on $X^\cmu_n$.

\medskip

As in \secref{parameter c}, if the group $G_{ad}$ is not simple, the line bundle
$\CP_{X^\cmu_n}$ is naturally a product of several line bundles, one for each 
simple factor. Hence, also in the present context we will interpret $c$
as an invariant symmetric bilinear form $\cLambda\otimes \cLambda\to \BC$.
As in {\it loc. cit.}, the categories $\fD\mod^{c'}(X^\cmu_n)$ 
and $\fD\mod^{c''}(X^\cmu_n)$ are equivalent if $c''-c'$ is integral.

\medskip

In order to introduce the category of factorizable sheaves, we will need
to define a particular object of the category $\fD\mod^c(X^\cmu_\emptyset)$,
denoted $\CL^{\cmu}_\emptyset$. 

Let $\oX^\cmu_\emptyset\subset X^\cmu_\emptyset$ be the open subscheme, corresponding
to divisors of the form $\underset{k}\Sigma\, \cmu_k\cdot y_k$
with all $y_k$ distinct and each $\cmu_k$ being the negative of one of the 
simple co-roots. Let $j^{Diag}$ denote the corresponding open embedding.

By \lemref{fiber of line bundle}, the line bundle $\CP_{X^\cmu_\emptyset}|_{\oX^\cmu_\emptyset}$ canonically
trivializes. Indeed, $(\cla,\cla+2\crho)_{Kil}=0$ whenever
$\cla$ is of the form $w(\crho)-\crho$ for some $w\in W$; in particular
for $\cla$ being a simple co-root.

Hence, the category $\fD\mod^c(\oX^\cmu_\emptyset)$ is canonically
the same as $\fD\mod(\oX^\cmu_\emptyset)$. 

\medskip

We let $\oL^\cmu_\emptyset\in \fD\mod^c(\oX^\cmu_\emptyset)$ to be the following
object. It corresponds via the above equivalence
$$\fD\mod^c(\oX^\cmu_\emptyset)\simeq \fD\mod(\oX^\cmu_\emptyset)$$
to the {\it sign} local system on $\oX^\cmu_\emptyset$. The latter is, by
definition, the product of sign local systems on each $\oX^{(m_\imath)}$
when we write $X^\cmu_\emptyset$ as in \eqref{sym power}.

\medskip

We define 
$$\CL^\cmu_\emptyset:=j^{Diag}_{!*}(\oL^\cmu_\emptyset).$$

Note that the Goresky-MacPherson extension is taken in the
category $\fD\mod^c(\oX^\cmu_\emptyset)$ (and {\it not} in
$\fD\mod(\oX^\cmu_\emptyset)$).

\medskip

\noindent{\it Example.} Let $G=SL_2$ and $c$ be irrational. Then the fibers
and co-fibers of $\CL^\cmu_\emptyset$ on the closed sub-variety 
$X^\cmu_\emptyset-\oX^\cmu_\emptyset$ are zero.

\medskip

By construction, the system of objects
$\cmu\mapsto \CL^\cmu_\emptyset$ has the following factorization property with
respect to \eqref{basic line bundle factorization}: for $\cmu=\cmu_1+\cmu_2$,
\begin{equation} \label{basic sheaf factorization}
\on{add}_{\cmu_1,\cmu_2,disj}^*(\CL^\cmu_\emptyset)
\simeq \CL^{\cmu_1}_\emptyset \boxtimes
\CL^{\cmu_2}_\emptyset.
\end{equation}
These isomorphisms are compatible with refinements of partitions.

\ssec{}   \label{intr factor sheaves}

We are now ready to introduce the sought-for category of factorizable sheaves.
We define $\wt{\FS}{}^c_n$ to have as objects (twisted) D-modules $\CL^\cmu_n\in \fD\mod^c(X^\cmu_n)$,
defined for each $\cmu\in \cLambda$, equipped with {\it factorization isomorphisms}:

For any partition $\cmu=\cmu_1+\cmu_2$ and the corresponding map 
$$\on{add}_{\cmu_1,\cmu_2,disj}:\left(X^{\cmu_1}_\emptyset \times X^{\cmu_2}_n\right)_{disj}\to X^\cmu_n,$$
we must be given an isomorphism
\begin{equation} \label{factorization for arb}
\on{add}_{\cmu_1,\cmu_2,disj}^*(\CL^\cmu_n)\simeq 
\CL^{\cmu_1}_\emptyset \boxtimes \CL^{\cmu_2}_n,
\end{equation}
compatible with refinements of partitions with respect to the isomorphism 
\eqref{basic sheaf factorization}.

A morphisms between two factorizable sheaves $^1\CL_n=\{{}^1\CL^\cmu_n\}$ and
$^2\CL_n=\{{}^2\CL^\cmu_n\}$ is a collection of maps $^1\CL^\cmu_n\to {}^2\CL^\cmu_n$,
compatible with the isomorphisms \eqref{factorization for arb}.

\medskip

Let $\oX^n\overset{j^{poles}}\hookrightarrow X^n$ be the complement to 
the diagonal divisor. By the same token, we define the category $\wt{\FS}{}^c_{\overset{\circ}n}$. We have a natural restriction functor
$(j^{poles})^*:\wt{\FS}{}^c_n\to \wt{\FS}{}^c_{\overset{\circ}n}$ and its right adjoint
$$(j^{poles})_*:\wt{\FS}{}^c_{\overset{\circ}n}\to 
\wt{\FS}{}^c_{n}.$$

Let now $\ol{n}$ be a partition $n=n_1+...+n_k$, and let 
$X^k\overset{\Delta_{\ol{n}}}\to X^n$ and $\oX^k
\overset{\overset{\circ}\Delta_{\ol{n}}}\to X^n$ be the
corresponding subschemes. We have the natural functors

$$(\Delta_{\ol{n}})_*:\wt{\FS}{}^c_{k}\to \wt{\FS}{}^c_n \text{ and }
(\overset{\circ}\Delta_{\ol{n}})_*:
\wt{\FS}{}^c_{\overset{\circ}k}\to \wt{\FS}{}^c_n.$$

The right adjoint functors are easily seen to be defined on the level
of derived categories (the latter are understood simply as the derived
categories of the corresponding abelian categories)
$$(\Delta_{\ol{n}})^!:D^+(\wt{\FS}{}^c_n)\to D^+(\wt{\FS}{}^c_{k})
\text{ and }
(\overset{\circ}\Delta_{\ol{n}})^!:D^+(\wt{\FS}{}^c_n)\to 
D^+(\wt{\FS}{}^c_{\overset{\circ}k}),$$
and coincide with the same-named functors on the level of
underlying twisted D-modules.

\ssec{}

We shall now introduce a (full, abelian) subcategory $\FS^c_n\subset \wt{\FS}{}^c_n$,
which will be our main object of study. An object $\CL_n\in \wt{\FS}{}^c_n$
belongs to $\FS^c_n$ if the following two conditions are satisfied:

\medskip

\noindent(i) Finiteness of support: $\CL^\cmu_n$ is non-zero only for $\cmu$
belonging to finitely many cosets $\cLambda/\on{Span}(\check\Delta)$.
For each such coset,
there exists $\ol\cnu=\cnu_1,...,\cnu_n\in \cLambda^n$, such that for
each $\cmu$ belonging to the above coset, the support of 
$\CL^\cmu_n$ is contained in the subscheme $X^\cmu_{n,\leq \ol\cnu}$.

\medskip

\noindent(ii) To state the second condition, we shall first do it "over" $\oX^n$, i.e.,
we will single out the subcategory $\FS^c_{\overset{\circ}n}$ inside
$\wt{\FS}{}^c_{\overset{\circ}n}$.

Our requirement is that there are {\it only finitely many} 
collections $(\cmu_1,...,\cmu_n)$,
such that for $\cmu=\underset{i=1,...,n}\Sigma\, \cmu_i$, the singular support
of $\CL^\cmu_{\overset{\circ}n}$,
viewed as a twisted D-module on 
\begin{equation} \label{part away}
X^\cmu_{n,\leq\ol\cnu}\underset{X^n}\times \oX^n
\end{equation}
(for some/any choice of $\ol\cnu$ such that the support condition is satisfied),
contains the conormal to the sub-scheme $\oX^n$, where the latter is embedded
into \eqref{part away} by means of $(x_1,...,x_n)\mapsto \Sigma\, \cmu_i\cdot x_i$.

\medskip

\noindent (Note that the above condition is actually a condition on $\cmu$:
when the latter is fixed, there are only finitely many partitions 
$\cmu=\underset{i=1,...,n}\Sigma\, \cmu_i$ with $\cmu_i\leq \cnu_i$
where $\ol\cnu$ bounds the support of our sheaf.)

\medskip

Now, condition (ii) over $X^n$ is that 
for any partition $n=n_1+...+n_k$ each of the cohomologies
of $(\overset{\circ}\Delta_{\ol{n}})^!(\CL_n)$, which is a priori
an object of $\wt{\FS}{}^c_{\overset{\circ}k}$, belongs in fact to
$\FS^c_{\overset{\circ}k}$.

\ssec{}

Let us fix the pole points $\ox=(x_1,...,x_n)\in X^n$,
and let $X^\cmu_\ox$ denote the fiber of $X^\cmu_n$ 
over this configuration. Proceeding as above, we
can introduce the categories $\wt{\FS}{}^c_\ox$ and 
$\FS^c_\ox$. We shall now describe some special
objects in them. With no restriction of generality we
can assume that the $x_i$'s are distinct.

\medskip

Let $\ol\cla=\cla_1,...,\cla_n$ be an $n$-tuple
of elements of $\cLambda$. For each $\cmu$
consider the corresponding closed subscheme
$X^\cmu_{\ox,\leq \ol\cla}:=X^\cmu_\ox \cap X^\cmu_{n,\leq \ol\cla}$.
If $\cmu=
\underset{i=1,...,n}\Sigma\, \cla_i-\underset{\imath\in \CI}\Sigma\,
m'_\imath\cdot \check\alpha_\imath$, then
\begin{equation} \label{trunc scheme}
X^\cmu_{\ox,\leq \ol\cla}\simeq \underset{\imath\in \CI}\Pi\,
X^{(m'_\imath)}.
\end{equation}

Let $$\oX^\cmu_{\ox,\leq \ol\cla}\overset{j^{Diag,poles}}\hookrightarrow
X^\cmu_{\ox,\leq \ol\cla}$$ be the open subscheme corresponding
to divisors of the form
$$\Sigma\, \cmu_k\cdot y_k+\underset{i=1,...,n}\Sigma\, \cla_i\cdot x_i, \text{  with   }
\Sigma\, \cmu_k+\underset{i=1,...,n}\Sigma\, \cla_i=\cmu$$
where all the $y_k$'s are pairwise distinct and different from the $x_i$'s,
and each $\cmu_k$ is the negative of a simple coroot. I.e., 
$\oX^\cmu_{\ox,\leq \ol\cla}$ is the complement of the diagonal
divisor in the product \eqref{trunc scheme}.
Note that the restriction to this subscheme of the line bundle $\CP_{X^\cmu_n}$ 
is constant with fiber
$$\underset{i=1,...,n}\bigotimes\, \omega_{x_i}^{(\cla_i,\cla_i+2\crho)_{\frac{Kil}{2}}}.$$
We define a local system on $\oX^\cmu_{\ox,\leq \ol\cla}$, denoted $\oL^\cmu_{\ox,\ol\cla}$,
as in the case of $\oL^\cmu_\emptyset$, using the product of sign local
systems on the factors in \eqref{trunc scheme}.

\medskip

Let $X^\cmu_{\ox,=\ol\cla}$ be the open subset of $X^\cmu_{\ox,\leq \ol\cla}$,
corresponding to divisors of the form 
$\Sigma\, \cmu_k\cdot y_k+\underset{i=1,...,n}\Sigma\, \cla_i\cdot x_i$
with $y_k\neq x_i$. We have the corresponding open embeddings:
$$\oX^\cmu_{\ox,\leq \ol\cla}\overset{'j^{Diag,poles}}\hookrightarrow
X^\cmu_{\ox,=\ol\cla}\overset{''j^{Diag,poles}}\hookrightarrow 
X^\cmu_{\ox,\leq \ol\cla}$$
with $j^{Diag,poles}={}''j^{Diag,poles}\circ {}'j^{Diag,poles}$.

\medskip

We let
\begin{align*}
&\CL^\cmu_{\ox,\ol\cla,!}:={}''j^{Diag,poles}_{!}\circ
{}'j^{Diag,poles}_{!*}(\oL^\cmu_{\ox,\ol\cla}),\,\,\,\,
\CL^\cmu_{\ox,\ol\cla,*}:={}''j^{Diag,poles}_{*}\circ
{}'j^{Diag,poles}_{!*}(\oL^\cmu_{\ox,\ol\cla})\\
&\text{ and }\CL^\cmu_{\ox,\ol\cla,!*}:={}''j^{Diag,poles}_{!*}\circ
{}'j^{Diag,poles}_{!*}(\oL^\cmu_{\ox,\ol\cla})\simeq
j^{Diag,poles}_{!*}(\oL^\cmu_{\ox,\ol\cla}).
\end{align*}

The collections $\CL_{\ox,\ol\cla,!}:=\{\CL^\cmu_{\ox,\ol\cla,!}\}$,
$\CL_{\ox,\ol\cla,!*}:=\{\CL^\cmu_{\ox,\ol\cla,!*}\}$, $\CL_{\ox,\ol\cla,*}:=\{\CL^\cmu_{\ox,\ol\cla,*}\}$
are naturally objects of $\wt{\FS}{}^c_\ox$.

\medskip

Points (a) and (c) of the following lemma essentially results from the definitions, whereas
point (b) follows from \cite{FS} (see \secref{proofs FS} below).

\begin{lem} \hfill \label{descr of irr in FS}

\smallskip

\noindent{\em(a)}
The objects $\CL_{\ox,\ol\cla,!*}$ are the irreducibles of $\wt{\FS}{}^c_\ox$.

\smallskip

\noindent{\em(b)}
The objects $\CL_{\ox,\ol\cla,!}$ and $\CL_{\ox,\ol\cla,*}$ are of finite length.

\smallskip

\noindent{\em(c)}
The cones of the natural maps
$$\CL_{\ox,\ol\cla,!}\to \CL_{\ox,\ol\cla,*!}\to \CL_{\ox,\ol\cla,*}$$
are extensions of objects of the form $\CL_{\ox,\ol\cla{}',!*}$ for
$\ol\cla'\leq \ol\cla$.

\end{lem}

Finally, we will use the following result, which also follows from
\cite{FS} (see \secref{proofs FS} below):

\begin{thm} \hfill \label{which objects}

\smallskip

\noindent{(a)} Assume that $c$ is rational. Then all the objects
$\CL_{\ox,\ol\cla,!},\CL_{\ox,\ol\cla,*!}$ and $\CL_{\ox,\ol\cla,*}$
belong to $\FS^c_\ox$.

\smallskip

\noindent{(b)}
Assume that $c$ is irrational. Then 

\begin{itemize}

\item(i)
The objects
$\CL_{\ox,\ol\cla,!}$ and $\CL_{\ox,\ol\cla,*}$ never belong to $\FS^c_\ox$.

\item(ii)
An object $\CL_{\ox,\ol\cla,!*}$ belongs to $\FS^c_\ox$ if and only if
all $\cla_i$ are dominant.

\item(iii) The category $\FS^c_\ox$ is semi-simple.

\end{itemize}

\end{thm}

\ssec{}

We can now formulate our main theorem:

\begin{thm}  \label{main}
Let $c$ be irrational. Then
there exists an equivalence of abelian categories
$$\Whit^c_n\to \FS^c_n.$$
\end{thm}

\ssec{}      \label{old FS}
Let us now explain the relation between our set-up and that of \cite{FS}. 
The difference is that in {\it loc. cit.} the authors work with D-modules
on the spaces $X^\cmu_{n,\leq \ol\cnu}$, rather than with twisted D-modules.

Let us again fix an $n$-tuple of distinct points $x_1,...,x_n$ and
identify the line bundle $\CP_{X^\cmu_n}|_{X^\cmu_{\ox,\leq \ol\cnu}}$
explicitly. Let 
$$\cmu-\underset{i=1,...,n}\Sigma\, \cnu_i=-\underset{\imath}\Sigma\, m_\imath\cdot \check\alpha_\imath$$
The scheme $X^\cmu_{\ox,\leq \ol\cnu}$ can be identified
with the corresponding product of symmetric powers,
$\underset{\imath}\Pi\, X^{(m_\imath)}$.

The line bundle $\CP_{X^\cmu}|_{X^\cmu_{\ox,\leq \ol\cnu}}$ is then
$$\CO_{X^\cmu_{\ox,\leq \ol\cnu}}\left(-\underset{\imath}\Sigma\, d_\imath\cdot \Delta_\imath-
\underset{\imath_1\neq \imath_2}\Sigma\, d_{\imath_1,\imath_2}\cdot 
\Delta_{\imath_1,\imath_2}-\underset{\imath,j=1,...,n}\Sigma\, d_{\imath,j}\cdot \Delta_{\imath,j}\right),$$
where $\Delta_\imath$ is the diagonal divisor on $X^{(m_\imath)}$,
$\Delta_{\imath_1,\imath_2}$ is the incidence divisor of $X^{(m_{\imath_1})}\times
X^{(m_{\imath_2})}$, $\Delta_{\imath,j}$ the incidence divisor on $X^{(m_\imath)}\times x_j$,
and
$$d_\imath=(\check\alpha_\imath,\check\alpha_\imath)_{\frac{Kil}{2}},\,\,
d_{\imath_1,\imath_2}=(\alpha_{\imath_1},\alpha_{\imath_2})_{Kil},\,\,
d_{\imath,j}=(\alpha_{\imath},\cnu_j)_{Kil}.$$

\medskip

Let us assume now that our curve $X$ is $\BA^1$ (as in \cite{FS}), with coordinate $t$.
We will denote by $t^{\imath}_1,...,t^{\imath}_{m_\imath}$ the
corresponding functions on $X^{m_\imath}$. The function
\begin{multline*}
\sff^{\cmu}_{\ol\cnu}:=\underset{\imath}\Pi\, \left(\underset{1\leq k_1,k_2\leq m_\imath}
\Pi\, (t^\imath_{k_1}-t^\imath_{k_2}) \right)^{2d_\imath}\cdot
\underset{\imath_1\neq \imath_2}\Pi\,
\left(\underset{1\leq k_1\leq m_{\imath_1}, 1\leq k_2\leq m_{\imath_2}}
\Pi\, (t^{\imath_1}_{k_1}-t^{\imath_2}_{k_2})\right)^{d_{\imath_1,\imath_2}}
\cdot \\
\cdot \underset{\imath,j=1,...,n}\Pi\, \left(\underset{1\leq k\leq m_\imath}\Pi\, 
(t^k_\imath-x_j)\right)^{d_{\imath,j}}
\end{multline*}
on $X^\cmu_{\ox,\leq \ol\cnu}$ trivializes the line bundle in question.
This allows to view the twisted D-modules $\CL^\cmu_\ox$ comprising
an object $\CL_\ox\in \wt{\FS}{}^c$ as plain D-modules.

However, the definition of the standard object, such as $\oL^\cmu_{\ox,\ol\cla}$
is more complicated. The latter equals the product of the sign local system and
the pull-back by means of the map
$$\oX^\cmu_{\ox,\leq \ol\cnu}\overset{\sff^{\cmu}_{\ol\cnu}}\to \BG_m$$
of the D-module $\Psi(c)$ on $\BG_m$.
Here $\Psi(c)$ is the Kummer D-module, generated by one section $"z^c"$
and satisfying the relation
\begin{equation} \label{Kummer}
z\partial_z\cdot "z^c"=c\cdot "z^c",
\end{equation}
where $z$ is a coordinate on $\BG_m$.

\ssec{} \label{proofs FS}
 
This subsection is included in order to navigate the reader in the
structure of the proofs of results from \cite{FS} that are used in this
paper. 

\medskip

To simplify the notation, let as assume that $n=1$, i.e., $\ox=\{x\}$. 
The main hard result that we
use is that the the objects $\CL_{x,\cla,!*}$ for $\cla\in \cLambda^+$
belong to $\FS^c_x$, i.e., they satisfy the singular support condition. 
This is established in Theorem II.8.18 of {\it loc. cit.} This theorem
amounts to an explicit calculation of vanishing cycles.

In fact, that theorem says that for any $\cla\in \cLambda$ and $c$,
the cotangent space to the point $\cmu\cdot x\in X^\cmu_x$
belongs to the singular support of $\CL^\cmu_{x,\cla,!*}$ (resp.,
$\CL^\cmu_{x,\cla,!}$, $\CL^\cmu_{x,\cla,*}$) if and only if
$\cmu$ appears as a weight of the irreducible module
(resp., Verma module, dual Verma module) of highest weight
$\cla$ over the corresponding quantum group.

\medskip

This implies point (b) of \lemref{descr of irr in FS}, as well as
points (a), (b,i), (b,ii) of \thmref{which objects}. 

\medskip

Let us now comment on how to deduce that $\FS^c_x$ is semi-simple
for $c$ irrational (we will just copy the argument from {loc. cit.}
Sect. III.18). The proof relies on the following (Lemma III.5.3 of {loc. cit.}):

\begin{lem} \label{hom computation}
Assume that both $\CL_1,\CL_2,\in \FS^c_x$ are supported
on $X^\cmu_x$ for $\cmu$ belonging to a single coset
in $\cLambda/\on{Span}(\check\Delta)$. Then there exists
$\cmu_0$ such that for all $\cmu\geq \cmu_0$ the map
$$\Hom_{\FS^c_x}(\CL_1,\CL_2)\to
\Hom_{\fD\mod^c(X^\cmu_x)}(\CL^\cmu_1,\CL^\cmu_2)$$
is an isomorphism.
\end{lem}

The proof of this lemma is not difficult, but we emphasize that it
uses condition (ii) that singles out $\FS^c_x$ inside $\wt{\FS}{}^c_x$ in 
an essential way.

\medskip

To prove the semi-simplicity of $\FS^c_x$ we have to show that 
for any $\cla_1,\cla_2\in \cLambda^+$,
$$Ext^1_{\FS^c_x}(\CL_{x,\cla_1,!*},
\CL_{x,\cla_2,!*})=0.$$
We distinguish two cases.

\medskip

\noindent Case 1: $\cla_1=\cla_2=:\cla$. An extension class gives rise
to an extension of twisted D-modules over $X^\cmu_{x\leq \cla}$.
$$0\to \CL^\cmu_{x,\cla,!*}\to \CM^\cmu\to \CL^\cmu_{x,\cla,!*}\to 0.$$
Applying \lemref{hom computation}, we obtain that for all $\cmu$ that are
sufficiently large, this extension is {\it non-split}. Restricting this
extension to the open subscheme $\oX^\cmu_{x,\leq\cla}$, we obtain
a non-trivial extension of the corresponding local systems. This, however,
contradicts the factorization property.

\medskip

\noindent Case 2: $\cla_1\neq \cla_2$. We will show that 
$$Ext^1_{\fD\mod^c(X^\cmu_x)}(\CL^\cmu_{x,\cla_1,!*},
\CL^\cmu_{x,\cla_2,!*})=0,$$
which would imply our assertion in view of \lemref{hom computation}.

Since the situation is essentially Verdier self-dual, we can assume that
$\cmu_2\geq \cmu_1$. The object $\CL^\cmu_{x,\cla_1,!*}$ 
is supported on the closed sub-scheme $X^\cmu_{x,\leq \cla_1}$.
We claim that both $!$- and $*$- restrictions of $\CL^\cmu_{x,\cla_2,!*}$
to this subscheme are zero. 

By factorization, it is enough to prove the latter assertion for $\cmu=\cla_1$,
in which case $X^\cmu_{x,\leq \cla_1}=\on{pt}$. The assertion is local,
so we can assume that $(x\in X)\simeq (0\in \BA^1)$, and pass from twisted D-modules
to usual D-modules, as in \secref{old FS}. Consider the action of
$\BG_m$ on $\BA^1$, and hence on $(\BA_1)^{\cmu_1}_{0,\leq \cla_2}$
We obtain that $\CL^\cmu_{x,\cla_2,!*}$ is monodromic against
the character sheaf $\Psi(c')$, where
$$c'=c\cdot \left((\cla_2,\cla_2+2\crho)_{\frac{Kil}{2}}-
(\cla_1,\cla_1+2\crho)_{\frac{Kil}{2}}\right).$$ However, the integer
$(\cla_2,\cla_2+2\crho)_{\frac{Kil}{2}}-(\cla_1,\cla_1+2\crho)_{\frac{Kil}{2}}$
is non-zero, and since $c\notin \BQ$, we obtain that $c'\notin \BZ$. Hence, the stalk and
co-stalk of this D-module at $0$ is zero.

\section{Zastava spaces}   \label{Zast}

\ssec{}
In order to prove \thmref{main} we have to be able to pass from (twisted) D-modules
on $\fW_n$ to (twisted) D-modules on $X^\cmu_n$. This will be done using
ind-schemes $\CZ^\cmu_n$ (defined for each $\cmu\in \cLambda$), that map
to both $\fW_n$ and $X^\cmu_n$, and that are called "Zastava spaces".

\medskip

For a coweight $\cmu$ let $\BunBm^\cmu$ denote the stack of $B^-$-bundles of 
degree $(2g-2)\crho-\cmu$. We will think of a point of $\BunBm^\cmu$ as a triple:

\begin{itemize}

\item A $G$-bundle $\fF_G$.

\item A $T$-bundle $\fF_T$ such that for any $\lambda\in \Lambda$ the degree
of the corresponding line bundle $\lambda(\fF_T)$ is $\langle \lambda,(2g-2)\crho-\cmu\rangle$.

\item A collection of surjective bundle maps
$$\kappa^{\lambda,-}:\CV^\lambda_{\fF_G}\to \lambda(\fF_T),$$
which satisfy the Pl\"ucker equations.

\end{itemize}

Let $\fp^-$ and $\fq^-$ denote the natural maps
$$\Bun_G \leftarrow \BunBm^\cmu\to \Bun_T,$$
respectively. Let $\CP_{\BunBm^\cmu}$ denote the line bundle
$\fp^-{}^*(\CP_{\Bun_G})$. Almost by definition we have:

\begin{lem}
The line bundle $\CP_{\BunBm^\cmu}$ is isomorphic to the
pull-back under $\fq^-$ of the line bundle $\CP_{\Bun_T}$
over $\Bun_T$ (see \secref{intr line bundle}).
\end{lem}

\ssec{}

We let $\CZ^\cmu_n$ denote the open sub-stack in the product
$\fW_n\underset{\Bun_G}\times \BunBm^\cmu$ that corresponds to
the condition that the composed (meromorphic) maps
\begin{equation} \label{composition of kappas}
\omega^{\langle \lambda,\crho\rangle}\overset{\kappa^\lambda}\to
\CV^\lambda_{\fF_G}\overset{\kappa^{\lambda,-}}\to \lambda(\fF_T)
\end{equation}
are non-zero. Let $'\fp^-$, $'\fp$ denote the projections $\CZ^\cmu_n\to \fW_n$
and $\CZ^\cmu_n\to \BunBm^\cmu$, respectively.

\medskip

Taking the zeroes/poles of the maps \eqref{composition of kappas}, 
we obtain a natural map
$$\pi^\cmu:\CZ^\cmu_n\to X^\cmu_n.$$

\ssec{}

The next three assertions repeat \cite{BFG}, Sect. 2.16
(see also \cite{BFGM}, Sect. 2 for a less abstract treatment):

\begin{prop}  \label{nature of Zastava}
Let $\{\ox,\fF_G,\fF_T,(\kappa^\lambda),(\kappa^{\lambda,-})\}$ be a point
of $\CZ^\cmu_n$, and let $D\in X^\cmu_n$ be its image under $\pi^\cmu$.
Then the restriction of $\fF_G$ to the open curve $X-\on{supp}(D)$
is canonically isomorphic to $\fF_G=\omegacrho\overset{T}\times G$,
with the tautological maps $\kappa^\lambda,\kappa^{\lambda,-}$.
\end{prop}

\begin{proof}
This is just the fact that the stack $N\backslash \oG/B^-$ is isomorphic
to $\on{pt}$, where $\oG$ denotes the open Bruhat cell in $G$.
\end{proof}

\begin{cor}
The (ind)-stack $Z^\cmu_n$ is in fact an (ind)-scheme.
\end{cor}

\begin{prop}  \label{factorization of Zastava}
For $\cmu=\cmu'+\cmu''$ there exists a canonical
isomorphism of stacks
\begin{equation} \label{Zast factor}
\left(X^{\cmu'}_\emptyset \times X^{\cmu''}_n\right)_{disj}
\underset{X^\cmu_n}\times \CZ^\cmu_n\simeq
\left(X^{\cmu'}_\emptyset \times X^{\cmu''}_n\right)_{disj}
\underset{(X^{\cmu'}_\emptyset \times X^{\cmu''}_n)}\times
\left(\CZ^{\cmu'}_\emptyset\times \CZ^{\cmu''}_n\right).
\end{equation}
\end{prop}

\begin{proof}

Let $D'$ and $D''$ be points of $X^{\cmu'}_\emptyset$ and
$X^{\cmu''}_n$, respectively, with disjoint supports. 

Objects classified by both the LHS and the RHS in \eqref{Zast factor}
are local in $X$. Thus, we can think of the LHS as defining
a certain data on $X-\on{supp}(D')$ and $X-\on{supp}(D'')$ separately,
with a gluing datum over $X-\left(\on{supp}(D')\cup \on{supp}(D'')\right)$.
We have to show that the gluing datum in question is in fact
superfluous, but this follows immediately from \propref{nature of Zastava}.

\end{proof}

\ssec{}

Let us make the following observation:
\begin{lem}  \label{line bundle on Zastava}
The line bundle $'\fp^-{}^*(\CP_{\fW_n})\simeq 
{}'\fp^*(\CP^{\otimes -1}_{\BunBm^\cmu})$
identifies canonically with $\pi^\cmu{}^*(\CP_{X^\cmu_n})$.
\end{lem}

\begin{proof}
This follows from the fact that the diagram
$$
\CD
\CZ^\cmu_n @>{'\fp}>> \BunBm^\cmu \\
@V{\pi^\cmu}VV   @V{\fq^-}VV  \\
X_n^\cmu @>{AJ}>>  \Bun_T
\endCD
$$
commutes, where the lower horizontal arrow is the Abel-Jacobi map
of \secref{intr line bundle}, i.e, it sends
a divisor $D$ to the $T$-bundle $\omegacrho(-D)$.
\end{proof}

This allows us to define the functors
$$D(\fD^c\mod(\fW_n))\to D(\fD\mod^c(X^\cmu_n)).$$
Let us denote by ${}'\fp^{-,\cmu,\cdot}$ the functor
$D(\fD\mod^c(\fW_n))\to D(\fD\mod^c(\CZ^\cmu_n))$ given by
$$\CF\mapsto ({}'\fp^-)^!(\CF)[-\on{dim.rel.}(\BunBm^\cmu,\Bun_G)].$$

\noindent(We note that this functor {\it essentially} commutes 
with Verdier duality, since the morphism $\fp^-:\BunBm^\cmu\to \Bun_G$ is smooth
for $\cmu$ such $\langle \alpha,\cmu\rangle <-(2g-2)$,
which is what we will be able to assume in practice.)

\medskip

Thus, we can consider the functor
$$\pi^\cmu_*\circ {}'\fp^{-,\cmu,\cdot}:
D(\fD\mod^c(\fW_n))\to D(\fD\mod^c(X^\cmu_n)).$$


\medskip

We will prove:
\begin{prop} \label{perversity}
For all $c$ and $\CF\in \Whit^c_n$ the object
$\pi^\cmu_*({}'\fp^{-,\cmu,\cdot})(\CF))$ is concentrated in the cohomological 
degree $0$.
\end{prop}

In addition, we will prove the following assertion:

\begin{thm}  \label{cleanness} 
For $c$ irrational and any $\CF\in \Whit^c_n$, the object
$$\pi^\cmu_!({}'\fp^{-,\cmu,\cdot})(\CF))\in D(\fD\mod^c(X^\cmu_n))$$
is well-defined, \footnote{The issue here is that the direct image
with compact supports in not a priori defined on non-holonomic
D-modules.} and
the natural morphism
$$\pi^\cmu_!({}'\fp^{-,\cmu,\cdot})(\CF))\to \pi^\cmu_*({}'\fp^{-,\cmu,\cdot}(\CF))$$
is an isomorphism.
\end{thm}

We emphasize that assertion of \thmref{cleanness} is false without 
the assumption that $c$ be irrational.

\medskip

\noindent{\it Remark.} The morphism $\pi^\cmu$ is affine, so for a (twisted) 
D-module $\CF'$ on $\CZ^\cmu_n$, the object of the derived category given
by $\pi^\cmu_*(\CF')$
lives in non-positive cohomological degrees, and the object $\pi^\cmu_!(\CF')$
lives in non-negative cohomological degrees. Hence, \thmref{cleanness} formally
implies \propref{perversity}. Nonetheless, we will give an independent 
proof of this proposition, because it holds without the assumption that $c$
be irrational.

\ssec{}

Our present goal is to establish a key factorization property of the D-modules on $\CZ^\cmu_n$ that are obtained 
from objects of $\Whit^c_n$ by means of ${}'\fp^{-,\cmu,\cdot}$:

\begin{prop}   \label{factor of Whit}
For $\CF\in \Whit^c_n$ and $\cmu=\cmu_1+\cmu_2$, under the isomorphism
of \eqref{Zast factor}, the D-module
$$\on{add}_{\cmu_1,\cmu_2,disj}^*\left({}'\fp^{-,\cmu,\cdot}(\CF)\right)\in 
\fD\mod^c\left(\left(X^{\cmu_1}_\emptyset \times X^{\cmu_2}_n\right)_{disj}
\underset{X^\cmu_n}\times \CZ^\cmu_n\right)$$
goes over to
$$'\fp^{-,\cmu_1,\cdot}(\CF_{\emptyset})\boxtimes {}'\fp^{-,\cmu_2,\cdot}(\CF)\in
\fD\mod^c\left(\left(X^{\cmu_1}_\emptyset \times X^{\cmu_2}_n\right)_{disj}
\underset{(X^{\cmu_1}_\emptyset \times X^{\cmu_2}_n)}\times
\left(\CZ^{\cmu_1}_\emptyset\times \CZ^{\cmu_2}_n\right)\right).$$
These isomorphisms are compatible with refinements of factorizations.
\end{prop}

\begin{proof}

Let us consider the following relative version of the stack $(\fW_n)_{\text{good at }\oy}$,
introduced in \secref{good at y}. Namely, let $(\fW_n)_{\text{good at }\cmu_1}$ be the open substack
of $X^{\cmu_1}_{\emptyset}\times \fW_n$, where a divisor $D\in X^{\cmu_1}_\emptyset$
is forbidden to hit the pole points $(x_1,...,x_n)\in X^n$, and the
$\kappa^\lambda$'s are bundle maps on a neighbourhood of $\on{supp}(D)$. 

Over $X^{\cmu_1}_{\emptyset}$ we have a group-scheme, denoted $\CN^\reg_{\cmu_1}$,
and a group ind-scheme $\CN^{\mer}_{\cmu_1}$; the latter is endowed with a 
character $\chi_{\cmu_1}:\CN^{\mer}_{\cmu_1}\to \BG_a$. 
Over $(\fW_n)_{\text{good at }\cmu_1}$
there is a $\CN^\reg_{\cmu_1}$-torsor, denoted $_{\cmu_1}\fW_n$. 
The total space of this torsor is acted on by
$\CN^{\mer}_{\cmu_1}$. 

Consider the action map
$$\on{act}_{\cmu_1}: \CN^{\mer}_{\cmu_1}\overset{\CN^\reg_{\cmu_1}}\times {}_{\cmu_1}\fW_n\to
(\fW_n)_{\text{good at }\cmu_1}.$$
From the definition of the Whittaker category it follows that for any $\CF\in \Whit_n$, we have:
$$\on{act}_{\cmu_1}^*(\CF)\simeq \chi_{\cmu_1}^*(exp)\boxtimes \CF.$$

\medskip

The pre-image of $(\fW_n)_{\text{good at }\cmu_1}$ under the map 
$$\left(X^{\cmu_1}_\emptyset \times X^{\cmu_2}_n\right)_{disj}
\underset{X^\cmu_n}\times \CZ^\cmu_n\overset{'\fp^-}\to X^{\cmu_1}_{\emptyset}\times \fW_n$$
goes over under the isomorphism \eqref{Zast factor} to the substack
\begin{equation} \label{good part}
\left(X^{\cmu_1}_\emptyset \times X^{\cmu_2}_n\right)_{disj}
\underset{(X^{\cmu_1}_\emptyset \times X^{\cmu_2}_n)}\times
\left(\oZ^{\cmu_1}_\emptyset\times \CZ^{\cmu_1}_n\right),
\end{equation}
where 
$\oZ^{\cmu_1}_\emptyset=\CZ^{\cmu_1}_\emptyset\underset{\fW_\emptyset}\times \fW_{\emptyset,0}$.

Note that by construction, we have a locally closed embedding of schemes over $X^{\cmu_1}_\emptyset$
$$\oZ^{\cmu_1}_\emptyset\to \CN^{\mer}_{\cmu_1}/\CN^\reg_{\cmu_1},$$
such that the pull-back of $\chi_{\cmu_1}^*(exp)$ identifies with the restriction of
$'\fp^{-,\cmu_1,\cdot}(\CF_{\emptyset})$ to this sub-scheme.

\medskip

For $\CF\in \Whit^c_n$, its pull-back onto the product 
$$(X^{\cmu_1}_{\emptyset}\times \fW_n)\underset{X^{\cmu_1}_\emptyset\times X^n}
\times (X^{\cmu_1}_\emptyset\times X^n)_{disj}$$
is the extension by * (and also by $!$) from $(\fW_n)_{\text{good at }\cmu_1}$. Hence,
it it sufficient to establish an isomorphism of twisted D-modules over the open sub-stack
appearing in \eqref{good part}. 

The assertion of the proposition follows now from the fact that 
the composition
$$\left(X^{\cmu_1}_\emptyset \times X^{\cmu_2}_n\right)_{disj}
\underset{(X^{\cmu_1}_\emptyset \times X^{\cmu_2}_n)}\times
\left(\oZ^{\cmu_1}_\emptyset\times \CZ^{\cmu_2}_n\right)\to
\left(X^{\cmu_1}_\emptyset \times X^{\cmu_2}_n\right)_{disj}
\underset{X^\cmu_n}\times \CZ^\cmu_n \to X^{\cmu_1}_\emptyset \times \fW_n$$
factors as
\begin{multline*}
\left(X^{\cmu_1}_\emptyset \times X^{\cmu_2}_n\right)_{disj}
\underset{(X^{\cmu_1}_\emptyset \times X^{\cmu_2}_n)}\times
\left(\oZ^{\cmu_1}_\emptyset\times \CZ^{\cmu_2}_n\right)\to \\
\to \left(X^{\cmu_1}_\emptyset \times X^{\cmu_2}_n\right)_{disj}
\underset{(X^{\cmu_1}_\emptyset \times X^{\cmu_2}_n)}\times
\left(\CN^{\mer}_{\cmu_1}/\CN^\reg_{\cmu_1}\times \CZ^{\cmu_2}_n\right) \simeq \\
\simeq 
\left(X^{\cmu_1}_\emptyset \times X^{\cmu_2}_n\right)_{disj}
\underset{(X^{\cmu_1}_\emptyset \times X^{\cmu_2}_n)}\times
\left(\CN^{\mer}_{\cmu_1}\overset{\CN^\reg_{\cmu_1}}\times 
\left({}_{\cmu_1}\fW_n\underset{\fW_n}\times \CZ^{\cmu_2}_n)\right)\right)
\to \\
\to \CN^{\mer}_{\cmu_1}\overset{\CN^\reg_{\cmu_1}}\times {}_{\cmu_1}\fW_n
\overset{\on{act}_{\cmu_1}}\to {}_{\cmu_1}\fW_n\to 
(\fW_n)_{\text{good at }\cmu_1}\hookrightarrow X^{\cmu_1}_\emptyset \times \fW_n,$$
\end{multline*}
where the second arrow is the isomorphism, following
from the trivialization of the $\CN^{\reg}_{\cmu_1}$-torsor
$$({}_{\cmu_1}\fW_n\underset{\fW_n}\times \CZ^{\cmu_2}_n)
\underset{X^{\cmu_1}_\emptyset \times X^{\cmu_2}_n}\times
\left(X^{\cmu_1}_\emptyset \times X^{\cmu_2}_n\right)_{disj},$$
see \propref{nature of Zastava}.

\end{proof}

\ssec{}

As a corollary of the above proposition, we obtain:

\begin{cor} \label{factor of image of Whit}
For $\CF\in \Whit^c_n$ and $\cmu=\cmu_1+\cmu_2$, we obtain:
$$\on{add}_{\cmu_1,\cmu_2,disj}^*\left(\pi^\cmu_*({}'\fp^{-,\cmu,\cdot}(\CF))\right)
\simeq
\pi^{\cmu_1}_*({}'\fp^{-,\cmu_1,\cdot}(\CF_{\emptyset}))\boxtimes
\pi^{\cmu_2}_*({}'\fp^{-,\cmu_2,\cdot}(\CF))$$
as objects of
$\fD\mod^c\left(X^{\cmu_1}_\emptyset \times X^{\cmu_2}_n\right)_{disj}$.
These isomorphisms are compatible with refinements of partitions.
\end{cor}

The next step is to analyze the object 
\begin{equation} \label{L prime}
'\CL^\cmu_\emptyset:=\pi^\cmu_*({}'\fp^{-,\cmu,\cdot}(\CF_{\emptyset}))
\in \fD\mod^c(X^{\cmu}_\emptyset).
\end{equation}
We shall prove:

\begin{thm} \label{F0} \hfill

\smallskip

\noindent(1) If $(\check\alpha_\imath,\check\alpha_\imath)_c\notin \BZ$
for any $\imath\in \CI$, then we a canonical isomorphism
$'\CL^\cmu_\emptyset\simeq \CL^\cmu_\emptyset$
{\it over} $\oX^\cmu_\emptyset$.

\smallskip

\noindent(2) If $c$ is irrational, the above isomorphism holds over
$X^\cmu_\emptyset$.

\medskip

Both isomorphisms are compatible
with the factorization isomorphisms.

\end{thm}

\section{Proofs--A}   \label{proofs A}

\ssec{Proof of \thmref{F0}(1)}

Let us first assume that $\cmu$ equals the negative of a simple co-root 
$\check\alpha_\imath$. The scheme
$\CZ_\emptyset^{\check\alpha_\imath}$ identifies canonically with $X\times \BG_a$,
and $\oZ_\emptyset^{\check\alpha_\imath}$ is the complement to the zero section,
corresponding to $0\in \BG_a$.

Recall that the line bundle $\CP_{X^{\check\alpha_\imath}_\emptyset}$ is canonically
trivial. However, over $\oZ_\emptyset^{\check\alpha_\imath}$ we have two trivializations
of the the corresponding line bundle: one inherited from that on 
$X^{\check\alpha_\imath}_\emptyset$, and the other from that on $\fW_{\emptyset, 0}$.
The discrepancy is given by the map
$$\oZ_\emptyset^{\check\alpha_\imath}\simeq X\times \BG_m\twoheadrightarrow \BG_m\overset{x\mapsto x^{2\cdot d_\imath}}
\longrightarrow \BG_m,$$
where $d_\imath$ is as in \secref{old FS}.

\medskip

Thus, $\pi^\cmu_*({}'\fp^{-,\cmu,\cdot}(\CF_{\emptyset}))$ 
is equal to the constant D-module
on $X$ times the vector space
$$H\left(\BG_m,\Psi(2\cdot d_\imath\cdot c)\otimes \exp\right),$$
where $\Psi(\cdot)$ is the Kummer D-module as in \eqref{Kummer}.

\medskip

Now, the "Gauss sum" formula, i.e., the canonical isomorphism
$$H(\BG_m,\Psi(2\cdot d_\imath\cdot c)\otimes \exp)\simeq \BC$$
is well-known from the theory of Fourier-Deligne transform.

\medskip

Let us now assume that $\cmu$ is arbitrary. \corref{factor of image of Whit}
together with the above computation, imply that the required isomorphism
holds after the pull-back to $\left(\underset{\imath}\Pi\, X^{m_\imath}\right)$,
away from the diagonal divisor, where $\cmu=-\underset{\imath}\Sigma\, 
m_\imath\cdot \alpha_\imath$.

We only have to show that the action of the symmetric group 
$\underset{\imath}\Pi\, \Sigma_{m_\imath}$ on the above constant sheaf
is given by {\it sign} character. But this follows from the K\"unneth formula,
as the cohomology $H(\BG_m,\Psi(2\cdot d_\imath\cdot c)\otimes \exp)$ is 
concentrated in degree $0$ and $\dim(\BG_m)=1$.

\qed

\ssec{Proof of \propref{perversity}}

Let us first show that for $\CF\in \Whit_n$, the object 
$$\pi^\cmu_*\circ {}'\fp^{-,\cmu,\cdot}(\CF)$$ is concentrated
in non-positive degrees. 

For that we can assume that 
$\cmu$ satisfies $\langle \alpha,\cmu\rangle <-(2g-2)$ for
all $\alpha\in \Delta^+$. Indeed, if
not, we can replace $\cmu$ by $\cmu_1=\cmu-k\cdot \crho$
with $k$ large enough, and then apply \corref{factor of image of Whit}
combined with \thmref{F0}(1) to factor the extra points away.

\medskip

For $\cmu$ as above the map $\fp^-:\BunBm^{\cmu}\to \Bun_G$
is smooth, and hence, so is the map $'\fp^-:\CZ^\cmu_n\to \fW_n$.
Hence, ${}'\fp^{-,\cmu,\cdot}(\CF)$ lives in cohomological degree $0$.
Hence, our assertion follows from the fact that the morphism
$\pi^\cmu$ is affine (see \cite{BFGM}, Sect. 5.1 for the proof of the latter fact).

\medskip

To prove that $\pi^\cmu_*\circ {}'\fp^{-,\cmu,\cdot}(\CF)$ is concentrated
in non-negative degrees we need to analyze the fibers of the map
$\pi^\cmu$.

\ssec{Analysis of the fibers}  \label{fibers}

Let us denote by $\CZ^\cmu_{loc,x}$ the fiber of $\CZ^\cmu_{1}$ over the point
$\cmu\cdot x\in X^{\cmu}_1$. For $\cla\in \cLambda$, let us denote by
$\CZ^{\cla,\cmu}_{loc,x}$ the pre-image in $\CZ^\cmu_{loc,x}$ of the sub-stack
$\fW_{x,\cla}$.  Let $\CP_{\CZ^\cmu_{loc,x}}$ and $\CP_{\CZ^{\cla,\cmu}_{loc,x}}$
denote the corresponding line bundles, obtained by restriction from $\CP_{\CZ^\cmu_1}$.


As is shown in \cite{BFGM}, Sect. 2.6, $\CZ^\cmu_{loc,x}$ identifies with a closed
sub-indscheme of the affine Grassmannian $\Gr_{G,x}=G(\CK_x)/G(\CO_x)$.

Let $S^\cla$ denote the $N(\CK_x)$-orbit of the point 
$t^\cla$, \footnote{We denote by $t$
a local parameter on the formal disc $\D_x$ around $x$;
$t^\cla\in \Gr_G$ is the projection of the point
in $G(\CK_x)$ corresponding to the map $\D^\times_x\overset{t}\to \BG_m(\CK_x)
\overset{\cla}\to T(\CK_x)\hookrightarrow G(\CK_x)$.}
and let
$S^{-,\cmu}$ denote the $N^-(\CK_x)$-orbit of the point $t^{\cmu}$. Then
$$\CZ^\cmu_{loc,x}\simeq S^{-,\cmu}\overset{T(\CO_x)}\times \omegacrho|_{\D_x} \text{ and }
\CZ^{\cla,\cmu}_{loc,x}\simeq \left(S^\cla\cap S^{-,\cmu}\right)
\overset{T(\CO_x)}\times \omegacrho|_{\D_x},$$
where by a slight abuse of notation we denote by $\omegacrho|_{\D_x}$ the corresponding
$T(\CO_x)$-torsor.

The line bundles $\CP_{\CZ^\cmu_{loc,x}}$ and $\CP_{\CZ^{\cla,\cmu}_{loc,x}}$ 
are induced from the canonical line bundle on $\Gr_{G,x}$ via  the above 
embeddings.

\medskip

By \propref{factorization of Zastava}, for a point $D\in X^\cmu_n$ given by 
$\Sigma\, \cmu'_k\cdot y_k+\Sigma\, \cmu''_j\cdot x_j$ 
with all the $y_k$'s and $x_j$'s pairwise distinct, its preimage in $\CZ^\cmu_n$
is isomorphic to the product
$$\underset{k}\Pi\, \CZ^{\cmu'_k}_{loc,y_k}\times 
\underset{j}\Pi\, \CZ^{\cmu''_j}_{loc,x_j}.$$

\medskip

By \propref{factor of Whit}, an object $\CF\in \Whit^c_n$ defines a (complex of)
twisted D-modules $\CF^\cmu_{loc,x}$ for every $\cmu$ and $x\in X$, so that
the !-restriction of ${}'\fp^{-,\cmu,\cdot}(\CF)$ to the fiber over the point
$D\in X^\cmu_n$ is isomorphic to the product
$$\left(\underset{k}\boxtimes\, (\CF_{\emptyset})^{\cmu'_k}_{loc,y_k}\right)\boxtimes
\left(\underset{j}\boxtimes\, \CF^{\cmu''_j}_{loc,x_j}\right).$$

\ssec{Analysis of the line bundle}    \label{an lin bdle}

By \lemref{line bundle on Zastava}, the line bundle 
$\CP_{\CZ^\cmu_{loc,x}}$ over $\CZ^\cmu_{loc,x}$ is canonically
{\it constant} with fiber 
\begin{equation} \label{mu line}
\omega_x^{-(\cmu,\cmu+2\crho)_{\frac{Kil}{2}}}.
\end{equation}

Note that the line bundle $\CP_{\fW_{x,\cla}}$ over $\fW_{x,\cla}$ is 
also constant with fiber 
\begin{equation} \label{la line}
\omega_x^{-(\cla,\cla+2\crho)_{\frac{Kil}{2}}}.
\end{equation}
Hence, the line bundle $\CP_{\CZ^{\cla,\cmu}_{loc,x}}$ is also
isomorphic to the constant line bundle with the above fiber.

\medskip

We obtain that $\CP_{\CZ^{\cla,\cmu}_{loc,x}}$ admits two
trivializations, defined up to a scalar. The discrepancy between them
is a function $\CZ^{\cla,\cmu}_{loc,x}\to \BG_m$, 
defined up to a multiplication by a scalar, that we shall denote by $\gamma^{\cla,\cmu}$. The following assertion will be used in
the sequel:

\begin{lem} \label{torus action}
The function $\gamma^{\cla,\cmu}$ intertwines the 
natural $T(\CO_x)$ action on $S^\cla\cap S^{-,\cmu}\subset 
\Gr_{G,x}$ and the action on $\BG_m$ given by the character
$$T(\CO_x)\to T \overset{(\cla-\cmu,\cdot)_{Kil}}\longrightarrow \BG_m.$$
\end{lem}

\ssec{}   \label{est from above}

Let us now return to the proof of \propref{perversity}. Consider the stratification
of $X^\cmu_n$ by means of the strata formed by divisors $\Sigma\, \cmu'_k\cdot y_k+
\Sigma\, \cmu''_j\cdot x_j$ with all the $y_k$'s and $x_j$'s pairwise distinct.

By \secref{fibers}, to prove the desired cohomological estimate, we have to show that for 
$\CF\in \Whit^c_{x}$ the following holds:
\begin{equation} \label{coh est}
\begin{cases}
&H^i\left(\CZ^\cmu_{loc,x}, \CF^{\cmu}_{loc,x}\right)=0 \text{ for } i<0 \text{ and any } \CF \\
&H^i\left(\CZ^\cmu_{loc,x}, (\CF_{\emptyset})^{\cmu}_{loc,x}\right)=0 \text{ for } 
i\leq 0
\text{ and } \CF=\CF_{\emptyset}
\end{cases}
\end{equation}
where by a slight abuse of notation, we view $\CF^{\cmu}_{loc,x}$ and 
$(\CF_{\emptyset})^{\cmu}_{loc,x}$ as non-twisted D-modules using
any trivialization of the line \eqref{mu line}.

\bigskip

To prove the first assertion in \eqref{coh est}, we can assume that $\CF$ is of the form
$\CF_{x,\cla,*}$ for some $\cla\in \cLambda^+$. In this case, the
$\CF^{\cmu}_{loc,x}\in D(\fD\mod^c(\CZ^\cmu_{loc,x}))$ is the *-extension of a 
complex of twisted D-modules on $\CZ^{\cla,\cmu}_{loc,x}$. 

Let $\oZ^\cmu_{x,\cla}$ denote the the pre-image of the locally closed substack
$\fW_{x,\cla}\subset \fW_x$. This is a smooth scheme, and the map
$$\pi^\cmu:\oZ^\cmu_{x,\cla}\to X^\cmu_{x,\leq \cla}$$
is flat. The complex $\CF^{\cmu}_{loc,x}$ is obtained from a lisse twisted
D-module on $\oZ^\cmu_{x,\cla}$ by !-restriction to the fiber over the point
$\cla\cdot x\in X^\cmu_{x,\leq \cla}$. Hence, it is concentrated in the cohomological
degrees $\geq \dim(X^\cmu_{x,\leq \cla})=\langle \rho,\cla-\cmu\rangle$.
However, since $\dim(\CZ^{\cla,\cmu}_{loc,x})=\dim(S^\cla\cap S^{-,\cmu})=
\langle \rho,\cla-\cmu\rangle$, our assertion follows.

\bigskip

To prove the assertion concerning $\CF_{\emptyset}$, we have to show, that the restriction
of $(\CF_{\emptyset})^{\cmu}_{loc,x}$ to the smooth part of $\CZ^{0,\cmu}_{loc,x}$,
which is a lisse D-module placed in the cohomological degree 
$\langle \rho,-\cmu\rangle$, is non-constant on each connected component,
where we are using the trivialization of the twisting obtained from 
trivializing the line \eqref{mu line}.

Let us describe this lisse D-module explicitly. It is the tensor product of
$(\gamma^{0,\cmu})^*(\Psi(c))$ and $\chi_x^0{}^*(exp)$,
where $\chi_x^0$ is the map
$$\CZ^{0,\cmu}_{loc,x}\simeq S^0\cap S^{-,\cmu}\to N(\CK_x)/N(\CO_x)\overset{\chi_x}\to
\BG_a.$$

\medskip

However, by \cite{FGV}, Prop. 7.1.7, coupled with \cite{BFGM}, Prop. 6.4,
it is known that the map $\chi^0_x$ is {\it non-constant}
on every irreducible component of $S^0\cap S^{-,\cmu}$. This implies that the
above tensor product is non-constant on every component, since the first factor
is tame, and the second is not.

\section{Proofs--B}              \label{sect F0}

\ssec{}

From now till the end of the paper we will assume that $c$ is irrational.
The goal of this section is to prove \thmref{F0}(2), as well as the following statement:

\begin{thm}  \label{other sheaves} 
For $c$ irrational there exists an isomorphism
$$\pi^\cmu_*\circ {}'\fp^{-,\cmu,\cdot}(\CF_{\ox,\ol\cla,!*})\simeq 
\CL^\cmu_{\ox,\ol\cla,!*}.$$
\end{thm}

The proofs of the two theorems are largely parallel.
We begin with the former.

\ssec{} 

By \propref{perversity}, the LHS 
(i.e., $'\CL^\cmu_\emptyset:=\pi^\cmu_*\circ {}'\fp^{-,\cmu,\cdot}(\CF_{\emptyset})$)
is a D-module, which coincides with the RHS (i.e., $\CL^\cmu_\emptyset$) over the open
sub-scheme $\oX^\cmu_\emptyset$. 

As a first step, we are going to show
that the $!$-restriction of the LHS to any stratum in
$X^\cmu_\emptyset-\oX^\cmu_\emptyset$ is concentrated in cohomological degrees
$\geq 1$. Applying factorization, this is equivalent to the fact that
\begin{equation} \label{coh on fiber}
H^i\Bigl((S^0\cap S^{-,\cmu}),\gamma^{0,\cmu}{}^*(\Psi(c))\otimes
\chi_x^0{}^*(exp)\Bigr),
\end{equation}
vanishes for $i=-|\cmu|+1$, whenever $|\cmu|>1$, where $|\cmu|$ denotes the
length of $\cmu$, i.e., $|\langle \rho,\cmu\rangle|$. 

\medskip

Since the dimension of every irreducible component of $S^0\cap S^{-,\cmu}$
is $|\cmu|$, it is enough to show that for every such component 
(or its dense open subset) $Y$,
\begin{equation}  \label{h y}
H^i\Bigl(Y,\gamma^{0,\cmu}{}^*(\Psi(c))\otimes
\chi_x^0{}^*(exp)\Bigr)
\end{equation}
vanishes for $i=-|\cmu|+1$. We shall now rewrite
the expression for the above cohomology. 

\ssec{}  \label{chi univ}

Consider the vector space $\fn/[\fn,\fn]\simeq \BA^r$, and let
$\chi_{x,univ}$ be the canonical map $N(\CK_x)\to \BA^r$, where $r$
is the semi-simple rank of $G$. Our character
$\chi_{x}$ can be taken to be the composition of $\chi_{x,univ}$ and {\it any} functional
$\ell:\BA^r\to \BA$, which is non-zero on all simple roots.

By the projection formula, 
$$
H^i\Bigl(Y,\gamma^{0,\cmu}{}^*(\Psi(c))\otimes
\chi_x^0{}^*(exp)\Bigr)\simeq
H^i\biggl(\BA^r,(\chi^0_{x,univ}|_Y)_!\left(\gamma^{0,\cmu}{}^*(\Psi(c))
\right)\otimes \ell^*(exp)\biggr).
$$

The scheme $S^0\cap S^{-,\cmu}$, and hence $Y$, is acted on by $T(\CO_x)$, 
and the map
$\chi^0_{x,univ}$ is equivariant, where $T(\CO_x)$ acts on $\BA^r$ via
the projection $T(\CO_x)\to T$ and the natural action of the latter on $\fn/[\fn,\fn]$.

By \lemref{torus action}, the map $\gamma^{0,\cmu}:S^0\cap S^{-,\cmu}\to \BG_m$ is 
$T(\CO_x)$-equivariant against the character $\mu:=(\cmu,?)_{Kil}$.
Hence, the complex
\begin{equation} \label{M mu}
\sM^Y:=(\chi^0_{x,univ}|_Y)_!\left(\gamma^{0,\cmu}{}^*(\Psi(c))\right)
\end{equation}
on $\BA^r$ is $T$-equivariant against the character sheaf
$\Psi(c\cdot \mu)$. \footnote{The latter is the pull-back of $\Psi(c)$
by means of the map $\mu:T\to \BG_m$.} In particular, $\sM^Y$ is lisse away
from the diagonal hyperplanes.

\medskip

Thus, the cohomology \eqref{h y}, {\it shifted by} $[r]$,
is the fiber at $\ell \in (\BA^r)^*$ of the Fourier-Deligne transform $\on{Four}(\sM^Y)$.
By the above equivariance property of $\sM^Y$, the complex $\on{Four}(\sM^Y)$
is also twisted $T$-equivariant, and hence is lisse away from the coordinate hyperplanes,
and in particular on a neighbourhood of $\ell$.



\ssec{}
We are ready now to return to the proof that the cohomology \eqref{h y}
vanishes in degree $1-|\cmu|$. 

\medskip

Let $\CI'\subset \CI$ be the minimal Dynkin sub-diagram, such that
$\cmu\in \on{Span}(\check\alpha_{\imath'},\, \imath'\in \CI')$,
and let $r'$ be its rank. We will distinguish two cases: (1) $r'=1$ 
and (2) $r' \geq1$.

\medskip

In case (1) $\cmu=(-m)\cdot \check\alpha_\imath$, where $\alpha_\imath$ is
the corresponding simple root. We can describe the intersection 
$S^0\cap S^{-,\cmu}$ explicitly. It is irreducible and isomorphic to
$\BA^{m-1}\times \BG_m$, with the map $\chi^0_{x,univ}$ being
$$\BA^{m-1}\times \BG_m\twoheadrightarrow \BA^{m-1}\to \BA^1\overset{\alpha_\imath}
\hookrightarrow \BA^r.$$
(Recall that by assumption $|\cmu|\geq 2$, hence, $m\geq 2$).
Since there are no non-constant maps $\BA^{m-1}\to \BG_m$,
the map  $\gamma^{0,\cmu}$ factors through the $\BG_m$-factor.
This implies that the cohomology \eqref{coh on fiber} vanishes
in all degrees.

\medskip

Let us now consider case (2). By \secref{chi univ}, it is enough to
show that $\sM^Y$ itself lives in perverse cohomological degrees
$\geq -\langle \rho,\cmu\rangle+2$.

The equivariance property of $\sM^Y$ against the character sheaf 
$\Psi(c\cdot\mu)$ on $T$, and since $c$ is irrational, implies that 
the complex $\sM^Y$ is the extension by zero from the complement 
to the union of coordinate hyperplanes in $\BA^{r'}$. However, over this
open subset, the fibers of the map $\chi^0_{x,univ}:Y\to \BA^{r'}$
have dimension $\leq \langle \rho,\cmu\rangle-r'$, and we are
done since $r'\geq 2$.









\ssec{}

Let us proceed with the proof of \thmref{F0}(2). We obtain, that the map
$$\pi^\cmu_*\circ {}'\fp^{-,\cmu,\cdot}(\CF_{\emptyset})\to
j_*^{Diag}(\oL^\cmu_\emptyset),$$
resulting from the isomorphism of \thmref{F0}(1), is injective.

\medskip

Since we are dealing with holonomic D-modules, the direct
images with compact supports are well-defined, and by 
Verdier duality, we obtain that the map
$$j_!^{Diag}(\oL^\cmu_\emptyset)\to 
\pi^\cmu_!\circ {}'\fp^{-,\cmu,\cdot}(\CF_{\emptyset})$$
is surjective.

\medskip

Consider the composition
$$j_!^{Diag}(\oL^\cmu_\emptyset)\twoheadrightarrow
\pi^\cmu_!\circ {}'\fp^{-,\cmu,\cdot}(\CF_{\emptyset})\to
\pi^\cmu_*\circ {}'\fp^{-,\cmu,\cdot}(\CF_{\emptyset})\hookrightarrow
j_*^{Diag}(\oL^\cmu_\emptyset),$$
where the middle arrow is the canonical map as in \thmref{cleanness}.
This map restricts to the tautological isomorphism over $\oX^\cmu_\emptyset$;
hence, it is the canonical map over the entire $X^\cmu_\emptyset$.

\medskip

Applying \thmref{cleanness} (which will be proven in the next section),
we deduce that
$$\pi^\cmu_!\circ {}'\fp^{-,\cmu,\cdot}(\CF_{\emptyset})\simeq
j_{!*}^{Diag}(\oL^\cmu_\emptyset)\simeq 
\pi^\cmu_*\circ {}'\fp^{-,\cmu,\cdot}(\CF_{\emptyset}),$$
as required.

\ssec{Proof of \thmref{other sheaves}}   \label{proof of other sheaves}

To simplify the notation, we will assume that $n=1$, i.e., $\ox$ consists of
one point $x$ (and $\ol\cla$ is just one co-weight $\cla$). The proof in the general
case is the same.

Both D-modules:
$$\pi^\cmu_*\circ {}'\fp^{-,\cmu,\cdot}(\CF_{x,\cla,!*})
\text{ and } \CL^\cmu_{x,\cla,!*}$$
are supported on the sub-scheme 
$X^\cmu_{x,\leq \cla}\subset X^\cmu_x$. By \propref{factor of Whit},
the desired isomorphism holds over the open sub-scheme
$\oX^\cmu_{x,\leq \cla}$. 

Moreover, by \thmref{F0}, the isomorphism in question
holds over a larger open sub-scheme: one consisting of divisors
$D=\underset{k}\Sigma\, \cmu_k\cdot y_k+ \cla\cdot x$ with
$y_k\neq x$. Thus, we have to show that the isomorphisms
holds also over this divisor. 

We will argue by induction on $\cla-\cmu$. The case $\cla=\cmu$
is evident. We assume that the assertion is true for all $\cmu'$ with
$|\cla-\cmu'|< |\cla-\cmu|$. Then, by factorization, the two
D-modules appearing in the theorem are isomorphic 
away from the point 
$\cmu\cdot x\overset{i_x^\cmu}\hookrightarrow X^\cmu_{x, \leq \cla}$.
Let us denote the corresponding open embedding 
$X^\cmu_{x,\leq \cla}-\{\cmu\cdot x\}\hookrightarrow X^\cmu_{x,\leq\cla}$
by $j^{pole}$.

\ssec{}   \label{other sheaves fibers}

First, we claim that the !-fiber of 
$\pi^\cmu_*\circ {}'\fp^{-,\cmu,\cdot}(\CF_{\ox,\ol\cla,!*})$
at the above point is concentrated in cohomological degrees $\geq 1$.

Indeed, the above fiber is given (up to a cohomological shift by $|\cla-\cmu|$) by 
\begin{equation} \label{prel fiber}
H^\bullet\left(\CZ^\cmu_{loc,x}, (\CF_{x,\cla,!*})^{\cmu}_{loc,x}\right). 
\end{equation}

\medskip

We need to show that the above cohomology vanishes in degree 
$-|\cla-\cmu|$. This is equivalent to the fact that the D-module
$(\CF_{x,\cla,!*})^{\cmu}_{loc,x}$ 
is non-constant on an open
part of each irreducible component of $\CZ^\cmu_{loc,x}$.

Replacing $\CZ^\cmu_{loc,x}$ by a dense open subset in the
support of the D-modules in question, we are reduced to the study of
$S^\cla\cap S^{-,\cmu}$, and the D-module 
$(\gamma^{\cla,\cmu})^*(\Psi(c))\otimes\chi_x^\cla{}^*(exp)$ on it,
where $\gamma^{\cla,\cmu}$ is as in \secref{an lin bdle}, 
and $\chi^\cla_x$ is the function $S^\cla\to \BG_a$, induced by the character
$\chi_x:N(\CK_x)\to \BG_a$, which is defined up to a shift. 

\medskip

Let us distinguish two cases: If the the function $\chi^\cla_x$ is non-constant
on the given irreducible component, the assertion follows as 
in the proof of \propref{perversity} (see the end of \secref{est from above}).

\medskip

If $\chi^\cla_x$ is constant, the sheaf in question is 
$(\gamma^{\cla,\cmu})^*(\Psi(c))$, and it is non-constant, since
$c$ is irrational, and hence $c\cdot (\cla-\cmu,\cdot)_{Kil}$
is not an integral character of $T$.

\ssec{}

The rest of the proof is similar to that of \thmref{F0}(2). Indeed, we
obtain that the map
$$\pi^\cmu_*\circ {}'\fp^{-,\cmu,\cdot}(\CF_{x,\cla,!*})\to
j_*^{pole}\left(j^{pole}{}^*(\CL^\cmu_{x,\cla,!*})\right)$$
is injective. Dually, we obtain a surjective map
$$j_!^{pole}\left(j^{pole}{}^*(\CL^\cmu_{x,\cla,!*})\right)\to
\pi^\cmu_!\circ {}'\fp^{-,\cmu,\cdot}(\CF_{x,\cla,!*}).$$

The composition
$$j_!^{pole}\left(j^{pole}{}^*(\CL^\cmu_{x,\cla,!*})\right)\to
\pi^\cmu_!\circ {}'\fp^{-,\cmu,\cdot}(\CF_{x,\cla,!})\to
\pi^\cmu_*\circ {}'\fp^{-,\cmu,\cdot}(\CF_{x,\cla,!})\to
j_*^{pole}\left(j^{pole}{}^*(\CL^\cmu_{x,\cla,!*})\right)$$
is the canonical map
$$j_!^{pole}\left(j^{pole}{}^*(\CL^\cmu_{x,\cla,!*})\right)\to
j_*^{pole}\left(j^{pole}{}^*(\CL^\cmu_{x,\cla,!*})\right),$$
because this is so over $X^\cmu_{x,\leq \cla}-\{\cmu\cdot x\}$.

\medskip

Now, \thmref{cleanness} implies the desired isomorphism
$$\pi^\cmu_!\circ {}'\fp^{-,\cmu,\cdot}(\CF_{x,\cla,!*})\simeq
\CL^\cmu_{x,\cla,!*}\simeq 
\pi^\cmu_*\circ {}'\fp^{-,\cmu,\cdot}(\CF_{x,\cla,!*}).$$

\section{Cleanness}   \label{proof of cleanness}

In this section we will prove \thmref{cleanness}.

\ssec{}

We introduce the stack $\BunBmb^\cmu$ (see \cite{BG}, Sect. 1.2), 
whose definition
is the same as that of $\fW_\emptyset$ (with $B$ replaced by $B^-$
and $\kappa^\lambda$ replaced by $\kappa^{\lambda,-}$) and the difference
being that we allow an arbitrary $T$-bundle of degree $(2g-2)\crho-\cmu$.

We have a tautological open embedding $\jmath^-:\BunBm^\cmu\hookrightarrow
\BunBmb^\cmu$, and the maps
$$\Bun^\cmu_T\overset{\fqbm}\longleftarrow  \BunBmb^\cmu
\overset{\fpbm}\longrightarrow \Bun_G,$$
that extend the corresponding maps for $\BunBm^\cmu$. 

Recall the line bundle $\CP_{\Bun_T}$ over $\Bun_T$ and the line bundle
$\CP_{\Bun_G}$ over $\Bun_G$. We let $\CP^T_{\BunBmb^\cmu}$ and
$\CP^G_{\BunBmb^\cmu}$, respectively, denote their pull-backs to 
$\BunBmb^\cmu$. The two are canonically isomorphic over the open
sub-stack $\BunBm^\cmu$. 

\ssec{}

We introduce the compactified Zastava space $\ol\CZ^\cmu_n$ as the open
sub-stack of $\fW_n\underset{\Bun_G}\times \BunBmb^\cmu$ corresponding
to the condition that all the compositions
$$\omega^{\langle\lambda,\crho\rangle}\overset{\kappa^\lambda}
\to \CV^\lambda_{\fF_G}\overset{\kappa^{-,\lambda}}\to \lambda(\fF_T)$$
are non-zero. We let $'\fpbm: \ol\CZ^\cmu_n\to \fW_n$
and $'\fp:\ol\CZ^\cmu_n\to \BunBmb^\cmu$ denote the corresponding base-changed
maps.

\medskip

As in the case of the usual Zastava spaces, we have a natural map
$$\ol\pi^\cmu:\ol\CZ^\cmu_n\to X^\cmu_n,$$
and an analog of \propref{factorization of Zastava} holds (with the same proof):

\begin{equation} \label{factor comp Zastava}
\left(X^{\cmu_1}_\emptyset \times X^{\cmu_2}_n\right)_{disj}
\underset{X^\cmu_n}\times \ol\CZ^\cmu_n\simeq
\left(X^{\cmu_1}_\emptyset \times X^{\cmu_2}_n\right)_{disj}
\underset{(X^{\cmu_1}_\emptyset \times X^{\cmu_2}_n)}\times
\left(\ol\CZ^{\cmu_1}_\emptyset\times \ol\CZ^{\cmu_2}_n\right).
\end{equation}

There is a tautological isomorphism of line bundles:
$$'\fp^*(\CP^T_{\BunBmb^\cmu})^{\otimes -1}\simeq 
(\ol\pi^\cmu)^*(\CP_{X^\cmu_n}).$$
We shall denote by $\fD\mod^c(\ol\CZ_n^\cmu)$ 
the corresponding category of twisted D-modules.

\medskip

Let us denote by $'\jmath^{-}$ the open embedding
$\CZ^\cmu_n\hookrightarrow \ol\CZ^\cmu_n$. For $\CF\in \Whit^c_n$
we can consider ${}'\fp^{-,\cmu,\cdot}(\CF)$ as an object of
$\fD\mod^c(\CZ_n^\cmu)$; this is
due to the identification of the line bundles $\CP^T_{\BunBmb^\cmu}$
and $\CP^G_{\BunBmb^\cmu}$ over $\Bun_{B^-}^\cmu$. We will
deduce \thmref{cleanness} from the next assertion:

\begin{thm} \label{thm cleanness}
Assume that $c$ is irrational. Then for $\CF\in \Whit^c_n$, the object
${}'\fp^{-,\cmu,\cdot}(\CF)$ is 
clean with respect to $'\jmath^{-}$, i.e., the map
$$'\jmath^{-}_!({}'\fp^{-,\cmu,\cdot}(\CF))\to {}'\jmath^{-}_*({}'\fp^{-,\cmu,\cdot}(\CF))$$
is an isomorphism in $\fD\mod^c(\ol\CZ_n^\cmu)$ (in particular, the LHS
is well-defined).
\end{thm}

\ssec{Proof of \thmref{cleanness}}

The basic observation is that the map $\ol\pi^\cmu$ is proper. Indeed,
$\ol\CZ^\cmu_n$ is a closed sub-scheme of the corresponding relative
(i.e., Beilinson-Drinfeld) version of the affine Grassmannian over $X^\cmu_n$.

Hence, it remains to notice that
$$\pi^\cmu_!({}'\fp^{-,\cmu,\cdot}(\CF))\simeq \ol\pi^\cmu_!
\left({}'\jmath^{-}_!({}'\fp^{-,\cmu,\cdot}(\CF))\right)
\text{ and }
\pi^\cmu_*({}'\fp^{-,\cmu,\cdot}(\CF))\simeq \ol\pi^\cmu_*\left(
{}'\jmath^{-}_*({}'\fp^{-,\cmu,\cdot}(\CF))\right).$$

\qed

\ssec{}

The rest of this section is devoted to the proof of \thmref{thm cleanness}. First, we 
will establish a cleanness-type result purely on $\BunBmb^\cmu$.

\medskip

Let $\CP^{ratio}_{\BunBmb^\cmu}$ denote the ratio of the two line bundles
$\BunBmb^\cmu$
$$\CP^{ratio}_{\BunBmb^\cmu}:={}^G\CP_{\BunBmb^\cmu}\otimes
\left({}^T\CP_{\BunBmb^\cmu}\right)^{\otimes -1}.$$

We note that the restriction of this line bundle to the open part
$\BunBm^\cmu$ is canonically trivial. 
(In fact, in \cite{BFG}, Theorem 11.6 it was
shown that, after passing from $\BunBmb^\cmu$ to its normalization,
the inverse of the corresponding section of $\CP^{ratio}_{\BunBmb^\cmu}$ is 
regular and its locus of zeroes is $\BunBmb^\cmu-\BunBm^\cmu$.)

We introduce $\fD\mod^{ratio,c}(\BunBmb^\cmu)$ as
the corresponding category of twisted D-modules on $\BunBmb^\cmu$.

\medskip

For a given $c$ let $\Const^c_{\BunBm^\cmu}$ denote canonical "constant"
object of the category $$\fD\mod^{ratio,c}(\BunBm^\cmu)\simeq 
\fD\mod(\BunBm^\cmu).$$ Our main technical tool is the following:

\begin{thm} \label{other cleanness}
For $c$ irrational, the object 
$\Const^c_{\BunBm^\cmu}\in \fD\mod^{ratio,c}(\Bun_{B^-}^\cmu)$ is clean with 
respect to $\jmath^{-}$, i.e., the maps
$$\jmath^{-}_!(\Const^c_{\BunBm^\cmu})\to 
\jmath^{-}_{!*}(\Const^c_{\BunBm^\cmu})\to
\jmath^{-}_*(\Const^c_{\BunBm^\cmu})$$
are isomorphisms in $\fD\mod^{ratio,c}(\BunBmb^\cmu)$.
\end{thm}

\ssec{Proof of \thmref{other cleanness}}

We will only sketch the proof, as it follows very closely the IC sheaf
computation in \cite{FFKM} or \cite{BFGM}, Sect. 5.

We will work with all connected components $\BunBmb^\cmu$,
so $\cmu$ will be dropped from the notation. We represent 
$\BunBmb$ as a union of open sub-stacks ${}^{\leq \cnu}\BunBmb$,
where the latter classifies the data $\{\fF_G,\fF_T,(\kappa^{-,\lambda})\}$,
where the length of the quotient $\lambda(\fF_T)/\on{Im}(\kappa^{-,\lambda})$
is $\leq \langle \lambda,\cnu\rangle$. Let ${}^{=\cnu}\BunBmb$ be the
closed sub-stack of ${}^{\leq \cnu}\BunBmb$, when the length of
the above quotient is exactly $\langle \lambda,\cnu\rangle$. 

We have
$$\BunBmb-\BunBm=\underset{\cnu\in \cLambda^{pos}-0}\bigcup\,
{}^{=\cnu}\BunBmb.$$

We will argue by induction and assume that the assertion of the theorem is
valid over $\BunBmb^{\leq \cnu'}$ for all $\cnu'$ with $|\cnu'|<|\cnu|$.

\medskip

Let $\CZ^{-,\cnu}$ be the corresponding Zastava space, i.e.,
an open subset of $\Bun_N\underset{\Bun_G}\times \BunBmb^\cnu$.
Let $\CP^{ratio}_{\CZ^{-,\cnu}}$ be the pull-back of 
$\CP^{ratio}_{\BunBmb^\cnu}$ under the natural projection. We let
$\fD\mod^{ratio,c}(\CZ^{-,\cnu})$ denote the corresponding category of
twisted D-modules. 

Let $\oZ^{-,\cnu}\overset{'\jmath^-}\hookrightarrow\CZ^{-,\cnu}$ be the open sub-scheme equal to the pre-image of 
$\BunBm\overset{\jmath^-}\hookrightarrow \BunBmb$.
One shows as in \cite{BFGM}, Sect. 3, that the problem of extension of the 
twisted D-module
$\Const^c_{\BunBm}$ from $\BunBm$ to $^{\leq \cnu}\BunBmb$ 
is equivalent to that of extension of the corresponding twisted D-module $\Const^c_{\oZ^{-,\cmu}}$
from $\oZ^{-,\cnu}$ to $\CZ^{-,\cnu}$.


\medskip

Consider the projection $\pi^{-,\cnu}:\CZ^{-,\cnu}\to X^\cnu$. Since
the pull-back of $\CP_{\Bun_G}$ to $\Bun_N$ is canonically trivial,
we have an identification of line bundles
$$\CP^{ratio}_{\CZ^{-,\cnu}}\simeq (\pi^{-,\cnu})^*(\CP_{X^\cnu}),$$
where $\CP_{X^\cnu}$ is the canonical line bundle over $X^\cnu$
defined as in \secref{intr line bundle}. In particular, we have a well-defined
direct image functors
$$\pi^{-,\cnu}_!,\pi^{-,\cnu}_*:D(\fD\mod^{ratio,c}(\CZ^{-,\cnu}))\to
D(\fD\mod^{ratio,c}(X^\cnu)).$$

As in \cite{BFGM}, Sect. 2.2.,
the map $\pi^{-,\cnu}$ admits a canonical section, denoted $\fs^{-,\cnu}$,
compatible with the above identification of line bundles. The image of 
$\fs^{-,\cnu}$ is the locus $^{=\cnu}\CZ^{-,\cnu}$ equal to the pre-image of 
$^{=\cnu}\BunBmb\subset \BunBmb$.

\medskip

By the induction hypothesis, the cleanness assertion for $\Const^c_{\oZ^{-,\cnu}}$
with respect to $'\jmath^{-}$ holds  away from the image of the section $\fs^{-,\cnu}$. 
Hence, to prove the theorem, it suffices to show that 
$$(\fs^{-,\cnu})^!\circ ({}'\jmath^{-})_{!*}(\Const^c_{\oZ^{-,\cmu}})=0.$$

However, since the morphism $\pi^{-,\cnu}$ is affine and there exists
a $\BG_m$-action along its fibers that contracts $\CZ^{-,\cnu}$ onto the 
image of $\fs^{-,\cnu}$, with $\Const^c_{\oZ^{-,\cmu}}$ being equivariant,
we have:
$$(\fs^{-,\cnu})^!\circ ({}'\jmath^{-})_{!*}(\Const^c_{\oZ^{-,\cnu}})
\simeq
(\pi^{-,\cnu}_!)\circ ({}'\jmath^{-})_{!*}(\Const^c_{\oZ^{-,\cnu}}):=\sK^{-,\cnu}.$$

\medskip

Again, by the induction hypothesis and factorization, $\sK^{-,\cnu}$
vanishes away from the main diagonal $X\subset X^\cnu$.
Moreover, by the defining property of the Goresky-MacPherson
extension, $\sK^\cnu$ lives in the cohomological degrees $>0$.

Hence, it suffices to see that its *-fiber at every point $x\in X$ lives
in the cohomological degrees $<0$. By base change and the
the induction hypothesis, the fiber in question is given by
$$H_c(S^0\cap S^{-,\cnu}, \Const^c_{\oZ^{-,\cnu}_{loc,x}}).$$

As in the proof of \propref{perversity}, the complex of D-modules
$\Const^c_{\oZ^{-,\cnu}_{loc,x}}$ is isomorphic to the pull-back
of $\Psi(c)$ by means of the function $\gamma^{0,\cnu}$, cohomologically
shifted by $[2|\cnu|]$.

\medskip

Since $\dim(S^0\cap S^{-,\cnu})=|\cnu|$, non-zero cohomology can only exist 
in degrees $\leq 0$ (and from the above we already know that it vanishes for 
degrees strictly less than $0$). 
Hence, it suffices to show that the above cohomology vanishes in degree $0$. 
The latter happens if and only if the lisse D-module, obtained by restricting
$\gamma^{0,\cnu}{}^*(\Psi(c))[2|\cnu|]$ to a smooth open part  
of every irreducible component of $S^0\cap S^{-,\cnu}$, is non-constant.

\medskip

However, from \lemref{torus action}, we know that 
the D-module in question is equivariant with respect to the $T(\CO_x)$-action 
on the scheme $S^0\cap S^{-,\cnu}$ against the character sheaf 
$\Psi(c\cdot \nu)$,
where $\nu\in \Lambda$ is $(\cnu,\cdot)_{Kil}$, and $c\cdot \nu\notin \BZ$ unless
$\cnu=0$, since $c$ is irrational.

\qed

\ssec{}

We shall now show how to deduce \thmref{thm cleanness} from 
\thmref{other cleanness}. We will use the following assertion,
which can be proved by the same argument as \cite{BG}, Sect. 5.3:

\begin{prop}  \label{ULA}
For every open sub-stack of finite type $U\subset \Bun_G$ and $\cnu,\cmu\in \cLambda^{pos}$ there exists $\cmu'\in \cLambda$ with $\cmu\geq \cmu'$,
such that the twisted D-module $\jmath^{-}_*(\Const^c_{\BunBm^{\cmu'}})$, 
restricted to $^{\leq \cnu}\BunBmb^{\cmu'}$ is ULA \footnote{See \cite{BG}, 
Sect. 5.1, for a review of the ULA
property.} with respect to the map $\fpbm$ over $U$.
\end{prop}

\medskip

Given an object $\CF\in \Whit_n^c$, let $\ol\cla=(\cla_1,...,\cla_n)$ be the bound on the 
order of the poles of the maps $\kappa^\lambda$ contained in the support of $\CF$ in $\fW_n$.

Consider now the support of
$'\jmath^{-}_*({}'\fp^{-,\cmu,\cdot}(\CF))$ 
on $\ol\CZ^\cmu_n$; denote it by $Y$; this is a stack of finite type.
The image of $Y$ in $\fW_n$ is contained in 
an open sub-stack $^{\leq \check\eta}\fW_{n,\leq \cla}$ of $\fW_{n,\leq \cla}$, where 
the total amount of zeroes of the maps $\kappa^\lambda$ does not exceed 
$\check\eta$. The image of $^{\leq \check\eta}\fW_{n,\leq \cla}$
in $\Bun_G$ under $\fp$ is contained in an open sub-stack of finite type that we 
will denote by $U$. Similarly, the image of $Y$ in $\BunBmb^{\cmu}$ is contained in 
$^{\leq \cnu}\BunBmb^{\cmu}$ for some $\cnu$. 

\medskip

Let us take $\cmu'$ with $\cmu\geq \cmu'$ given by \propref{ULA} for the above 
$U$ and $\cnu$.
Consider the open sub-scheme $^U\ol\CZ^{\cmu'}_n$ of $\ol\CZ^{\cmu'}_n$ equal to 
the pre-image of $U$ under the forgetful map $\CZ^{\cmu'}_n\to \Bun_G$.
We claim that it it sufficient to show that the cleanness
statement holds over $^U\ol\CZ^{\cmu'}_n$. 

Indeed, by factorization (i.e., \eqref{factor comp Zastava}) we can complement
any given point of $\ol\CZ^\cmu_n$ by points in $\oZ^{\cmu_k}_{loc,y_k}$ with
$\cmu-\cmu'=\underset{k}\Sigma\, \cmu_k$, and $y_k$ being away from the
support of the divisor equal to the image of our point under $\ol\pi^\cmu$,
to get a point of $\ol\CZ^{\cmu'}_n$. The image of this new point in
$\fW_n$ will still be contained in $^{\leq \check\eta}\fW_{n,\leq \cla}$; 
and hence its image in $\Bun_G$ will be contained in  $U$.

\medskip

Let us note that the line bundle $'\fp^*(\CP^T_{\BunBmb^{\cmu'}})^{\otimes -1}$ over
$\ol\CZ^{\cmu'}_n$ is the tensor product
$$('\fpbm)^*(\CP_{\fW_n}) \otimes{}'\fp^*(\CP^{ratio}_{\BunBmb^{\cmu'}}).$$
Hence, for $\CF'\in \fD^c\mod(\fW_n)$ and 
$\CF''\in \fD\mod(\BunBmb^{\cmu'})^{ratio,c}$
it makes sense to consider
$$\CF'\underset{\Bun_G}\boxtimes \CF''\in D(\fD\mod^c(\ol\CZ^{\cmu'}_n)).$$

\medskip

Now, the assertion of \thmref{thm cleanness} follows from the following
general statement:

\ssec{}
Let $\CY$ be a smooth scheme (or stack), $f:\CY'\to \CY$ a map,
and $\jmath:\oCY{}'\hookrightarrow \CY'$ and open sub-stack.
Let $\CF\in D(\fD\mod(\oCY{}'))$ be an object which is clean with respect
to $\jmath$, i.e., such that 
$$\jmath_!(\CF)\to \jmath_*(\CF)$$
is an isomorphism. 

Let now $Z\to \CY$ be a map and let $\CL$ be an object of
$D(\fD\mod(Z))$. Set
$$\oZZ{}':=Z\underset{\CY}\times \oCY{}' \text{ and } Z':=Z\underset{\CY}\times \CY',$$
and let $\jmath'$ be the corresponding open embedding
$\oZZ{}'\hookrightarrow Z'$.

\begin{lem}
Assume that $\jmath_*(\CF)$ is ULA with respect to $f$. Then 
$\CL\underset{\CY}\boxtimes \CF\in D(\fD\mod(\oZZ{}'))$
is clean with respect to $\jmath'$, i.e.,
$$\jmath'_!(\CL\underset{\CY}\boxtimes \CF)\to
\jmath'_*(\CL\underset{\CY}\boxtimes \CF)$$
is an isomorphism.
\end{lem}

\section{Equivalence}   \label{equivalence}

\ssec{}

In this section we shall prove \thmref{main}. Thus, we have 
to construct a functor
$$\sG_n:\Whit^c_n\to \wt{\FS}{}^c_n;$$
show that its image belongs to $\FS^c_n$, and prove that the above functor is
an equivalence.

\medskip

The first step has been essentially carried out already: for $\CF\in \Whit^c_n$,
we set
$$\sG_n(\CF)^\cmu:=\pi^\cmu_*({}'\fp^{-,\cmu,\cdot}(\CF))\in \fD\mod^c(X^\cmu_n).$$
These D-modules satisfy the required factorization property by 
\corref{factor of image of Whit} and \thmref{F0}. The functor $\sG$ is exact
by \propref{perversity}.

\ssec{}

Let us fix (distinct) points $\ox:=(x_1,...,x_n)$, and consider the corresponding functor
$$\sG_\ox:\Whit^c_\ox\to \wt{\FS}{}^c_\ox.$$

We will first prove:
\begin{thm}  \label{equiv at x}
The functor $\sG_\ox$ has its image in $\FS^c_\ox$ and induces an equivalence
with the latter sub-category.
\end{thm}

\begin{proof}

By \propref{descr of irr in Whit} and \thmref{other sheaves}, we obtain
that $\sG_\ox$ does indeed send $\Whit^c_\ox$ to $\FS^c_\ox$,
is faithful, and defines a bijection on the level of irreducible objects.

Thus, applying \thmref{which objects}(b), we obtain that the assertion
of \thmref{equiv at x} reduces to the statement that the category 
$\Whit^c_\ox$ (with $c$ irrational) is semi-simple. By \lemref{on stratum},
there are no non-trivial self-extensions of
$\CF_{\ox,\ol\cla,!}$. Hence, semi-simplicity of $\Whit^c_\ox$ is
equivalent to the next statement:

\begin{prop}   \label{Whit clean}
For $c$ irrational and any $\ol\cla\in \cLambda^+$, the maps
$$\CF_{\ox,\ol\cla,!}\to \CF_{\ox,\ol\cla,!*}\to \CF_{\ox,\ol\cla,*}$$
are isomorphisms.
\end{prop}

\end{proof}

\ssec{Proof of \propref{Whit clean}}

For simplicity we shall assume that $n=1$, i.e., $\ox=\{x\}$ and $\ol\cla=\{\cla\}$.
The proof in the general case is the same.

\medskip

Let $\CF'$ be the kernel of the map $\CF_{x,\cla,!}\to \CF_{x,\cla,!*}$.
It is sufficient to show that $\sG_\ox(\CF')=0$. For $\cmu\in \cLambda$,
consider the corresponding short exact sequence:
\begin{equation} \label{SOS to split}
0\to \pi^\cmu_*\circ {}'\fp^{-,\cmu,\cdot}(\CF')\to
\pi^\cmu_*\circ {}'\fp^{-,\cmu,\cdot}(\CF_{x,\cla,!})\to
\pi^\cmu_*\circ {}'\fp^{-,\cmu,\cdot}(\CF_{x,\cla,!*})\to 0.
\end{equation}

As in \secref{proof of other sheaves}, we will argue by induction
on $|\cla-\cmu|$ that $\pi^\cmu_*\circ {}'\fp^{-,\cmu,\cdot}(\CF')=0$.
The base of the induction $\cmu=\cla$ trivially holds. The induction
hypothesis and factorization imply that 
$\pi^\cmu_*\circ {}'\fp^{-,\cmu,\cdot}(\CF')$ is supported at the
point $\cmu\cdot x\in X^\cmu_{x,\leq\cla}$. 

\medskip

Let $i^\cmu_x$ denote the embedding of this point into
$X^\cmu_{x,\leq\cla}$. It is sufficient to show that the $0$-th
cohomology of $(i^\cmu_x)^!\left(\pi^\cmu_*\circ {}'\fp^{-,\cmu,\cdot}(\CF')\right)$
vanishes.

However, by \thmref{which objects}(b), the short exact sequence
\eqref{SOS to split} is split. So, it is sufficient to show that the $0$-th
cohomology of 
$(i^\cmu_x)^!\left(\pi^\cmu_*\circ {}'\fp^{-,\cmu,\cdot}(\CF_{x,\cla,!})\right)$
vanishes. But the latter calculation has been performed in 
\secref{other sheaves fibers}.

\qed

\ssec{}

We are now ready to show that the functor $\sG_n$ maps $\Whit^c_{n}$ to 
$\FS^c_{n}$.

\medskip

Recall the categories $\FS^c_{\overset{\circ}n}\subset \wt{\FS}{}^c_{\overset{\circ}n}$,
(see \secref{intr factor sheaves}).
Let $\Whit^c_{\overset{\circ}n}$ be the corresponding Whittaker category
over $\oX^n$. We have a functor
$$\sG_{\overset{\circ}n}:\Whit^c_{\overset{\circ}n}\to \wt{\FS}{}^c_{\overset{\circ}n}:$$

The following results from \thmref{which objects}(b) and
and \thmref{equiv at x}:

\begin{lem} \hfill

\smallskip

\noindent(a) Every object of $\Whit^c_{\overset{\circ}n}$ is isomorphic to a 
direct sum
$\underset{\ol\cla}\oplus\, \CF_{\overset{\circ}n,\ol\cla,!*}\otimes \CM(\ol\cla)$,
where $\CF_{\overset{\circ}n,\ol\cla,!*}$ is a relative (over $\oX^n$) version of
$\CF_{\ox,\ol\cla,!*}$, and $\CM(\ol\cla)$ is a D-module on $\oX^n$.

\smallskip

\noindent(b) Every object of $\FS^c_{\overset{\circ}n}$ is isomorphic to a 
direct sum
$\underset{\ol\cla}\oplus\, \CL_{\overset{\circ}n,\ol\cla,!*}\otimes \CM(\ol\cla)$,
where $\CL_{\overset{\circ}n,\ol\cla,!*}$ is a relative (over $\oX^n$) version of
$\CL_{\ox,\ol\cla,!*}$ and $\CM(\ol\cla)$ is a D-module on $\oX^n$.

\smallskip

\noindent(c) The functor $\sG_{\overset{\circ}n}$ induces an equivalence
\begin{equation} \label{G open}
\Whit^c_{\overset{\circ}n}\simeq \FS^c_{\overset{\circ}n}.
\end{equation}
\end{lem}

\medskip

For any partition $\ol{n}$: $n=n_1+...+n_k$ we have the mutually adjoint functors
$$(\Delta_{\ol{n}})^!,(\Delta_{\ol{n}})_*:D^+(\Whit^c_n)\leftrightarrows
D^+(\Whit^c_{\overset{\circ}k}),$$
which \footnote{Here again $D(\cdot)$ is understood as the derived category
of the abelian category, but it is easy to see that the above functors
induce the usual functors on the level of underlying twisted D-modules.}
are intertwined by means of $\sG$ with the corresponding functors
$$(\Delta_{\ol{n}})^!,(\Delta_{\ol{n}})_*:D^+(\wt{\FS}{}^c_n)\leftrightarrows
D^+(\wt{\FS}{}^c_{\overset{\circ}k}).$$

\medskip

The commutative diagrams 
$$
\CD
\Whit^c_n @>{(\Delta_{\ol{n}})^!}>> D^+(\Whit^c_{\overset{\circ}k}) \\
@V{\sG_n}VV  @V{\sG_{\overset{\circ}k}}VV \\
\wt{\FS}{}^c_n @>{(\Delta_{\ol{n}})^!}>> D^+(\wt{\FS}{}^c_{\overset{\circ}k})
\endCD
$$
and \eqref{G open} imply that $\sG_n$ indeed sends $\Whit^c_n$ to $\FS^c_n$,
as required.

\ssec{}

Finally, let us show that $\sG_n$ is an equivalence.

Let $\Whit^c_{Diag(n)}\subset \Whit^c_n$ and $\FS^c_{Diag(n)}\subset \FS^c_n$
be the subcategories of objects, consisting of twisted D-modules that
supported over the diagonal divisor $X^{Diag(n)}\subset X^n$. By induction
on $n$ we can assume that $\sG_n$ induces an equivalence
\begin{equation} \label{G diag}
\sG_{Diag(n)}:\Whit^c_{Diag(n)}\simeq \FS^c_{Diag(n)}.
\end{equation}

\medskip

Let $i_{Diag(n)}$ denote the morphism $X^{Diag(n)}\to X^n$. It induces
the functors 
$$(i_{Diag(n)})_*:\Whit^c_{Diag(n)}\to \Whit^c_n \text{ and }
\FS^c_{Diag(n)}\to \FS^c_n,$$
which are intertwined by $\sG$. The same is true for the functors
$$(j^{poles})_*:\Whit^c_{\overset{\circ}n}\to \Whit^c_n \text{ and }
\FS^c_{\overset{\circ}n}\to \FS^c_n.$$

Hence, in order to prove the theorem, it suffices to show the following:
for every $\CF\in \Whit^c_{\overset{\circ}n}$ there exists an inverse family of objects 
$\CF_i\in \Whit^c_{\overset{\circ}n}$ with $(j^{poles})^*(\CF_i)\simeq \CF$, such that for
any $\CF'\in \Whit^c_{Diag(n)}$ the direct limits
\begin{equation} \label{van lim}
\underset{i}{\underset{\longrightarrow}{lim}}\,
\Ext^k_{\Whit^c_n}\left(\CF_i,(i_{Diag(n)})_*(\CF')\right) \text{ and }
\underset{i}{\underset{\longrightarrow}{lim}}\,\Ext^k_{\FS^c_n}
\left(\sG_n\left(\CF_i\right),
\sG_n\left((i_{Diag(n)})_*(\CF')\right)\right)
\end{equation}
vanish for $k=0,1$. Note that in both cases, the corresponding $\Hom$
and $\Ext^1$ can be computed inside the ambient category of D-modules.

\medskip

Consider the pro-object in the category of D-modules on $\fW_n$, given by
$(j^{poles})_!(\CF)$. It is easy to see that it can be represented as 
the limit of an inverse family of objects from $\Whit^c_{\overset{\circ}n}$.
We take $\CF_i$ to be this family. 

\medskip

With this choice, the vanishing of the first limit in \eqref{van lim} is automatic.
The vanishing of the second limit follows from \thmref{cleanness} and 
the $(\pi^\cmu_!,\pi^\cmu{}^!)$ adjunction.

\end{document}